\documentclass[a4paper,reqno]{amsart}
\usepackage[utf8]{inputenc}
\usepackage{lmodern,setspace}
\usepackage{amsmath}
\usepackage{booktabs,enumitem}
\usepackage{amssymb,amsthm,mathrsfs,amsfonts,url}
\usepackage{graphicx}
\usepackage{subfigure}
\usepackage[subfigure]{ccaption}
\usepackage{multicol}
\usepackage{tikz,pgf}
\usetikzlibrary{arrows,automata,shapes,decorations.pathmorphing}
\usepackage{fancyhdr}
\usepackage{wrapfig}

\theoremstyle{plain}
\newtheorem{theorem}{Theorem}
\newtheorem{proposition}{Proposition}
\newtheorem{corollary}{Corollary}

\theoremstyle{definition}

\newtheorem{example}{Example}

\renewcommand{\thesubfigure}{\thefigure \alph{subfigure}} \makeatletter \renewcommand{\p@subfigure}{} \renewcommand{\@thesubfigure}{\bf Figure \thesubfigure:\hskip\subfiglabelskip} \makeatother

\newcommand{\np}[1]{\textbf{\textit{#1}}}
\newcommand{\Cc}[0]{\mathcal{C}}
\newcommand{\Aut}[1]{\textnormal{Aut}\left(#1\right)}
\newcommand{\hol}[2]{\mathcal{O}\left(#1,#2\right)}
\newcommand{\ana}[1]{\mathcal{A}\left(#1\right)}
\newcommand{\holC}[1]{\mathcal{O}\left(#1\right)}
\newcommand{\dk}[3]{d_{#1}^K(#2,#3)}

\newcommand{\op}[1]{\textnormal{#1}}
\newcommand{\dsp}[0]{ds_{\textnormal{supp}}^2}
\newcommand{\dspp}[0]{ds_{\textrm{supp}}'^2}
\newcommand{\lap}[0]{\lambda_\textnormal{supp}}
\newcommand{\lapp}[0]{\lambda_\textnormal{supp}'}
\newcommand{\pr}[0]{\textnormal{pr}}
\newcommand{\Dr}[1]{\DD_{#1}}
\newcommand{\De}[0]{\mathbb{D}}

\newcommand{\pC}[0]{\CC^\ast}
\newcommand{\ppC}[0]{\CC^{\ast\ast}}
\newcommand{\cp}[1]{\CC\PP^{#1}}
\newcommand{\CC}[0]{\mathbb{C}}
\newcommand{\DD}[0]{\mathbb{D}}
\newcommand{\HH}{\mathbb{H}}
\newcommand{\RR}{\mathbb{R}}
\newcommand{\NN}{\mathbb{N}}
\newcommand{\PP}{\mathbb{P}}
\newcommand{\dif}[1]{\textnormal{d}#1}
\newcommand{\ie}[0]{\textnormal{i}}
\DeclareMathOperator\artanh{artanh}

\setlist[enumerate]{label=\textnormal{(\alph*)},leftmargin=*,topsep=12pt,partopsep=12pt,itemsep=1mm}

\title{The Ahlfors lemma and Picard's theorems}
\author{Aleksander Simoni\v{c}}
\address{Fakulteta za matematiko in fiziko \\ Univerza v Ljubljani \\ Jadranska ulica 19 \\ 1000 Ljubljana}
\email{aleksander.simonic@student.fmf.uni-lj.si}

\begin{document}

\begin{abstract}
The article introduces Ahlfors' generalization of the Schwarz lemma. With this powerful geometric tool of complex functions in one variable, we are able to prove some theorems concerning the size of images under holomorphic mappings, including the celebrated Picard's theorems. The article concludes with a brief insight into the theory of Kobayashi hyperbolic complex manifolds.
\end{abstract}

\maketitle

\section{Introduction}

Since 1881, when famous French mathematician \textbf{Henry Poincar\'{e}} (1854--1912) connected hyperbolic geometry on a disc with complex analysis, it had been known that the \emph{Poincar\'{e} metric} has constant Gauss curvature $-1$. This is also why the metric is traditionally considered \emph{hyperbolic}. Surprisingly, it was not until 1938 that a Finnish mathematician \textbf{Lars V.~Ahlfors} (1907--1996), one of the first two Fields medalists, realized that the \emph{Schwarz-Pick lemma} (Theorem \ref{izr:sch.pick}) was a consequence of the negative curvature of the Poincar\'{e} metric. His result is known as the \emph{Ahlfors lemma} (Theorem \ref{izr:ahlfors}) or the \emph{Schwarz-Pick-Ahlfors lemma} in full. According to Ahlfors himself, he published the paper \cite{Ahl} because the lemma provides a relatively simple proof of the \emph{Bloch theorem} (Theorem \ref{thm:bloch}) with a very good estimation of the constant $B$ from the Bloch theorem.

It seems rather interesting that although the Ahlfors lemma is comparatively old and  straightforward, it is rarely presented in general textbooks on complex analysis of one variable, Ahlfors' classic itself \cite{AhlCplx} being no exception. On the other hand, it is included in a book \cite{Nar01} by Narasimhan, who proves and uses it further to prove Picard's theorems. That said, he does not prove the Bloch theorem, one of Ahlfors' own applications, and only briefly addresses the importance of completeness of metrics.

There are two reasons for writing this article. The first reason is to provide a proof of the original version of Ahlfors' lemma and then use it to prove Bloch's and Picard's theorems. This is presented in the way that is useful for the second reason, which is an introduction to the theory of hyperbolic complex manifolds. The article can thus serve as a motivation for that interesting and still developing area of complex geometry.

Let $\CC$ be a complex plane, $\Dr{r}:=\{z\in\CC\colon|z|<r\}$ an open disc with radius $r>0$ and $\Dr{r}^\ast:=\Dr{r}\setminus\{0\}$. Also $\DD:=\Dr{1}$ and $\DD^\ast:=\Dr{1}^\ast$. \emph{A domain} $\Omega\subseteq\CC$ is, by definition, an open and connected set. The family of holomorphic functions on the domain $\Omega$ is denoted by $\holC{\Omega}$, and the family of holomorphic mappings from a domain $\Omega_1$ to domain $\Omega_2$ by $\hol{\Omega_1}{\Omega_2}$.

\begin{theorem}[Schwarz's lemma]
\label{schlema}
Assume $f\in \hol{\DD}{\DD}$ and $f(0) = 0$.
\begin{enumerate}
\item Then $|f(z)| \leq |z|$ for every $z\in\DD$ and $|f'(0)| \leq 1$.
\item If $|f'(0)| = 1$, or if $|f(z_0)| = |z_0|$ for some $z_0\in\DD^\ast$, then there exists $a\in \partial\DD$ such that $f(z) = az$.
\end{enumerate}
\end{theorem}

In the years 1869--1870, the German mathematician {\textbf{Hermann A.~Schwarz}} (1843--1921) was trying to offer an ultimate proof of the celebrated Riemann mapping theorem. In his thesis (1851), Riemann proposed a theorem which stated that every simply connected domain $\Omega\subset\CC$ is biholomorphic to $\DD$. R\'{e}sum\'{e} of the second chapter of Schwarz's lecture entitled ``Zur Theorie der Abbildung'' \cite[pp.~109-111]{Schw} is: \emph{Let $f\colon\DD\to f(\DD)\subset\CC$ be a biholomorphic mapping with $f(0)=0$. Assume that $\rho_1$ (resp.~$\rho_2$) is the minimum (resp.~the maximum) distance from $0$ to the boundary of $f(\DD)$. Then $\rho_1|z|\leq|f(z)|\leq \rho_2|z|$ and $\rho_1\leq|f'(0)|\leq \rho_2$}. Schwarz derived the first inequality from examining real part of $\log(f(z)/z)$ and the second inequality from the Cauchy integral representation theorem for $f'(0)$. It is clear that part (a) of Theorem \ref{schlema} follows from this result. The present form, proof and name of the lemma were written in 1912 by a Greek-German mathematician {\textbf{Constantin Carath\'{e}odory}} (1873--1950). He popularized the lemma through various problems of conformal mappings. For proof of the lemma, Carath\'{e}odory used the maximum principle on an auxiliary function $f(z)/z$. Since $f(0)=0$, this function is holomorphic on $\DD$.

The Schwarz lemma could be very easily reformulated for discs of arbitrary radii. Supposing that $f\in\hol{\Dr{r_1}}{\Dr{r_2}}$, where $r_1,r_2>0$ and $f(0)=0$. Then the mapping $F(z):=r_2^{-1}f(r_1z)$ meets the conditions of the lemma, so we get $|F(z)|\leq|z|$ for $z\in\DD$. Then $|f(z)|\leq (r_2/r_1)|z|$ for $z\in\Dr{r_1}$. Let $f$ be the entire function, i.e. holomorphic on $\CC$ such that $f(\CC)\subset\Dr{r_2}$ for a fixed $r_2>0$. The radius $r_1$ can be arbitrary large, so we get $f(z)\equiv0$. This is the content of the following well-known theorem.

\begin{theorem}[Liouville]
Every bounded entire function is constant.
\end{theorem}

The connection between the Schwarz lemma and the Liouville theorem is a wonderful and simple example of the \emph{Bloch principle: nihil est in infinito quod non prius fuerit in finito}, which can be translated as \emph{there is nothing in the infinite which was not first in the finite}. In consequence, for a global result like Liouville's, there must be a more powerful local result, such as Schwarz's. A French mathematician \textbf{Andr\'{e} Bloch} (1893--1948) published his principle in the paper in 1926. Let $\Dr{r}(z_0):=\{z\in\CC\colon|z-z_0|<r\}$ be an open disc with radius $r>0$ and centre $z_0\in\CC$. We denote by $\ana{\Omega}$ the set of all continuous functions on $\overline\Omega$ which are holomorphic on $\Omega$. In 1924, Bloch proved

\begin{theorem}
\label{thm:bloch}
There is a universal constant $B>0$ with the property that for every value of $0<R<B$, every function $f\in\ana{\DD}$ with $|f'(0)|=1$ maps a domain $\Omega\subset\DD$ biholomorphically onto $\Dr{R}(z_0)$ for some $z_0\in f(\DD)$.
\end{theorem}

We have named the discs from the theorem \emph{simple (``schlicht'') discs}. The Bloch theorem is interesting because it guarantees the existence of simple discs with a fixed radius in the image of ``quite a large family'' of holomorphic functions on a disc. In accordance with his principle, Bloch derived the following celebrated global result from his ``local'' theorem.

\begin{theorem}[Little Picard theorem]
\label{izr:mali.pic}
Any entire function whose range omits at least two distinct values is a constant.
\end{theorem}

The above theorem is a remarkable generalization of the Liouville theorem. It is simple to find entire functions whose range is the entire $\CC$; nonconstant polynomials, for instance. The exponential function is an example of an entire function whose range omits only one value, namely zero. But there does not exist a nonconstant entire function whose range omits $0$ and $1$. The latter statement is actually equivalent to the Little Picard theorem since $(b-a)z+a$ is a biholomorphic mapping between $\CC\setminus\{0,1\}$ and $\CC\setminus\{a,b\}$, where $a\neq b$. Theorem \ref{izr:mali.pic} was proved in 1879 by \textbf{Charles \'{E}.~Picard} (1856--1941), by using arguments based on the \emph{modular function}. A modular function $\lambda(z)$ is a covering map from the upper halfplane $\HH:=\{z\in\CC\colon \Im(z)>0\}$ onto $\CC\setminus\{0,1\}$. The function $g(z):=(\ie-z)(\ie+z)^{-1}$ maps $\HH$ biholomorphically onto $\DD$. If $f$ is an entire function whose range omits $0$ and $1$, then $f$ can be lifted by $\lambda$ to $\widetilde{f}\in\hol{\CC}{\HH}$, i.e. $\lambda\circ\widetilde{f}=f$. Since $g\circ\widetilde{f}$ is constant according to the Liouville theorem, $\widetilde{f}$ is constant and therefore $f$ is also constant. The nontrivial and technically challenging part of the proof is the construction of such $\lambda$. One can find a construction in \cite[\S7.3.4]{AhlCplx}, where Theorem \ref{izr:mali.pic} is proved in that way. This is why mathematicians searched for ``elementary'' proofs that avoid modular function.

The name of Theorem \ref{izr:mali.pic} suggests that a similar theorem named after Picard exists.

\begin{theorem}[Big Picard theorem]
\label{izr:vel.pic}
In the neighborhood of an isolated essential singularity a holomorphic function takes every value in $\CC$ infinitely often with no more than one exception.
\end{theorem}

Similarly to the relation between the Liouville theorem and the Little Picard theorem, there is a weaker and more accessible theorem in the case of the Big Picard theorem. We know that a holomorphic function on $\Omega\setminus\{a\}$ has in $a$ one and only one type of isolated singularities: \emph{removable singularity}, \emph{pole} and \emph{essential singularity}. In the latter case, the limit $\lim_{z\to a}|f(z)|$ does not exist and this happens if and only if the image of the neighborhood of the point $a$ is dense in $\CC$. This proposition is also known as the Sohocki-Casorati-Weierstrass theorem \cite[Proposition 2.4.4]{BerGay1991}. Theorem \ref{izr:vel.pic} can be reformulated as a \emph{meromorphic extension: if a holomorphic function in the neighborhood of an isolated essential singularity omits two distinct values, then the singularity is removable or it is a pole}. In this case, a function becomes meromorphic.

The article is self-contained, very little of elementary complex analysis is assumed. Starting with the definition of the Poincar\'{e} metric on a disc, we calculate the corresponding distance and state the Schwarz-Pick lemma. We then prove some properties of inner distances, the most important of which is the \emph{Hopf-Rinow theorem} \ref{izr:HR}. Next we prove Ahlfors' lemma and give some applications of it: proof of Bloch's theorem, \emph{Landau's theorems} \ref{thm:landau1} and \ref{thm:landau2}, \emph{Schottky's theorem} \ref{izr:schottky} and Picard's theorems. We conclude the article with some properties of hyperbolic complex manifolds, especially those connected with Picard's theorems.

\section{The Poincar\'{e} metric on a disc}

In this section the Poincar\'{e} metric on a disc is introduced and the corresponding distance is calculated in order to apply the Schwarz-Pick lemma.

Introduce $\RR^+:=\{x\in\RR\colon x>0\}$ and $\RR_0^+:=\RR^+\cup\{0\}$. \np{The Poincar\'{e} metric} on $\Dr{r}$ is
\begin{equation}
\label{ena:poi}
d\rho_r^2 := \frac{4r^2|\dif{z}|^2}{(r^2 - |z|^2)^2}.
\end{equation}
This is a form of a \np{Hermitian pseudometric}, which is on domain $\Omega\subseteq\CC$ defined by
\begin{equation}
\label{hermet}
ds_{\Omega}^2 := 2\lambda(z)|\dif{z}|^2
\end{equation}
where $\lambda(z)\in\Cc^2(\Omega,\RR_0^+)$ is twice real-differentiable function with $\lambda(z)=\overline{\lambda(z)}$ and $Z(\lambda):=\{z\in\Omega\colon\lambda(z)=0\}$ is a discrete set. If $Z(\lambda)=\varnothing$, then $ds_{\Omega}^2$ is said to be a \np{Hermitian metric}. We can observe that \eqref{ena:poi} is really a Hermitian metric. For the sake of simplicity, let us say $d\rho^2:=d\rho_1^2$.

Let $\Omega_1$ and $\Omega_2$ be two domains on $\CC$ and $f\in\hol{\Omega_1}{\Omega_2}$. \np{The pullback} of arbitrary pseudometric $ds_{\Omega_2}^2$ is defined by $f^\ast(ds_{\Omega_2}^2):=2\lambda(f(z))|f'(z)|^2|\dif{z}|^2$, which is a pseudometric on $\Omega_1$. Poincar\'{e} noticed that for \emph{M\"{o}bius transformations}
\[
     \varphi_{a}(z):=\frac{z-a}{1-\bar{a}z},
\]
where $a\in\DD$, there is $\varphi_{a}\in\Aut{\DD}$ and $\varphi_{a}^\ast(d\rho^2)=d\rho^2$. By $\Aut{\DD}$ a family of holomorphic automorphisms of a disc is denoted. Therefore M\"{o}bius transformations are isometries for the Poincar\'{e} metric. In 1884 Poincar\'{e} proved that
\[
     \Aut{\DD} = \{e^{\ie a}\varphi_{b}(z)\colon a\in\RR,b\in\DD\}.
\]
This follows from the Schwarz lemma \cite[Examples 2.3.12]{BerGay1991}.

A pseudodistance can always be assigned to a Hermitian pseudometric. The process is described in what follows. A pseudodistance differs from the distance in metric spaces only in that the distance between two different points might be zero.

Let $\Omega\subseteq\CC$ be an arbitrary domain and $x,y\in\Omega$ arbitrary points. The mapping $\gamma\colon[0,1]\rightarrow\Omega$ is called \emph{$\Cc^n$-path from $x$ to $y$} for $n\geq0$ if $\gamma(t)$ is $n$-times differentiable mapping on $(0,1)$ and $\gamma(0)=x,\gamma(1)=y$. In the case $n=0$ we speak about $\Cc$-paths. \emph{The concatenation of $\Cc^n$-paths} $\gamma_1$ from $x$ to $y$ and $\gamma_2$ from $y$ to $z$ is $\Cc$-path
\[
  (\gamma_1\ast\gamma_2)(t) :=
  \left\{ \begin{array}{rl}
  \gamma_1(2t), & t\in[0,1/2] \\
  \gamma_2(2t-1), & t\in[1/2,1]
  \end{array}
  \right.
\]
from $x$ to $z$. \emph{Piecewise $\Cc^n$-path} $\gamma$ from $x$ to $y$ is $\gamma:=\gamma_1\ast\cdots\ast\gamma_k$ where $\gamma_1,\ldots,\gamma_k$ are $\Cc^n$-paths and $\gamma(0)=x,\gamma(1)=y$.

Assume that  domain $\Omega$ is equipped with a Hermitian pseudometric $ds_\Omega^2$. Let $\gamma\colon[0,1]\rightarrow\Omega$ be a piecewise $\Cc^1$-path from $x$ to $y$. \emph{The length} of $\gamma$ is defined by
\[
     L_{ds_\Omega^2}(\gamma) := \int_0^1 \sqrt{2(\lambda\circ\gamma)}|\dot{\gamma}| \dif{t}.
\]
The pseudodistance between the points is defined by $d_\Omega(x,y):=\inf L_{ds_\Omega^2}(\gamma)$, where the infimum goes through all piecewise $\Cc^1$-paths $\gamma$ from $x$ to $y$.

Let domains $\Omega_1,\Omega_2\subseteq\CC$ be  equipped with pseudometrics $ds_{\Omega_1}^2=2\lambda_1(z)|\dif{z}|^2$ and $ds_{\Omega_2}^2=2\lambda_2(z)|\dif{z}|^2$. Let there be points $x,y\in\Omega_1$, $f\in\hol{\Omega_1}{\Omega_2}$ and a piecewise $\Cc^1$-path $\gamma\colon[0,1]\rightarrow\Omega_1$ from $x$ to $y$. Assume that $f^\ast(ds_{\Omega_2}^2)\leq ds_{\Omega_1}^2$. Then $f(\gamma)$ is a piecewise $\Cc^1$-path from $f(x)$ to $f(y)$. We have
\begin{equation}
\label{ena:razdalja}
d_{\Omega_2}(f(x),f(y))\leq \int_0^1 \sqrt{2(\lambda_2\circ f\circ\gamma)}|f'(\gamma)||\dot{\gamma}|\dif{t} \leq \int_0^1\sqrt{2(\lambda_1\circ\gamma)}|\dot{\gamma}|\dif{t}.
\end{equation}
Because this is valid for every such path, it follows
\begin{equation}
\label{ena:skrcitev.met}
d_{\Omega_2}(f(x),f(y))\leq d_{\Omega_1}(x,y).
\end{equation}
If $f\in\hol{\Omega_1}{\Omega_2}$ is a biholomorphic mapping and $f$ is an isometry for pseudometrics i.e. $f^\ast(ds_{\Omega_2}^2)=ds_{\Omega_1}^2$, then we can, with similar inequality as \eqref{ena:razdalja}, but on inverse mapping $f^{-1}$, obtain $d_{\Omega_2}(f(x),f(y))= d_{\Omega_1}(x,y)$. In this case $f$ is also an isometry for the induced pseudodistances.

To the Poincar\'{e} metric on a disc we can explicitly write down the distance function between arbitrary points $p,q\in\DD$. We denote it with $\rho(p,q)$ and call it the \np{Poincar\'{e} distance}. The proposition below will show that it is expressible with an \emph{area hyperbolic tangent} \[
      \artanh{(x)}:=\frac{1}{2}\log{\frac{1+x}{1-x}}, \; x\in(-1,1).
\]
This function is increasing with zero at $x=0$. It is also $\lim_{x\to-1}\artanh(x)=-\infty$ and $\lim_{x\to1}\artanh(x)=\infty$.

\begin{proposition}
For arbitrary points $p,q\in\DD$, the Poincar\'{e} distance is
\begin{equation}
\label{equ:poi}
\rho(p,q) = \log{\frac{|1 - \bar{p}q| + |p-q|}{|1 - \bar{p}q| - |p-q|}} = \log{\frac{1+|\varphi_p(q)|}{1-|\varphi_p(q)|}} = 2\artanh{|\varphi_p(q)|}. \end{equation}
\end{proposition}

\begin{proof}
The second and third equalities are clear from the definitions, so the first equality remains to be proved.

Bearing in mind that rotations and M\"{o}bius transformations are isometries for the Poincar\'{e} metric, it is sufficient to show that \begin{equation}
\label{equ:poi.vmes}
\rho(0,a)=\log{\frac{1+a}{1-a}}
\end{equation}
for every $a\in[0,1)$. Because
\[
    \rho\left(0,\left|\frac{p-q}{1-\bar{p}q}\right|\right) = \rho\left(\varphi_{-p}(0),\varphi_{-p}\left(\frac{p-q}{1-\bar{p}q}\right)\right) = \rho(p,q),
\]
equation \eqref{equ:poi} follows from \eqref{equ:poi.vmes}. Let $\gamma(t):=x(t)+\ie y(t)$ be a piecewise $\Cc^1$-path from $0$ to $a$ and $\bar\gamma(t):=at$. Then
\begin{flalign*}
L_{d\rho^2}(\gamma)=\int_0^1 \frac{2\sqrt{\dot{x}^2(t)+\dot{y}^2(t)}}{1-x^2(t)-y^2(t)}\dif{t} &\geq \int_0^1 \frac{2\dot{x}(t)\dif{t}}{1-x^2(t)} \\
&= \log\frac{1+x(t)}{1-x(t)}\Bigl|_{t=0}^{t=1}=\log\frac{1+a}{1-a}=L_{d\rho^2}(\bar\gamma).
\end{flalign*} The inequality above becomes equality if and only if $y\equiv0$. If $L_{d\rho^2}(\gamma)=\rho(0,a)$, then $L_{d\rho^2}(\bar\gamma)=\rho(0,a)$, which is equivalent to \eqref{equ:poi.vmes}.
\end{proof}

The above proof makes it evident that the shortest path in the Poincar\'{e} metric from $0$ to $a\in[0,1)$ is a chord between those points. We call the shortest path in arbitrary metric \emph{a geodesic}. Using a proper rotation and M\"{o}bius transformation, we map this chord into the unique geodesic between arbitrary points on a disc. On Figure \ref{fig:disc1} we observe some geodesics through $0$ and on Figure \ref{fig:disc2} an action of M\"{o}bius transformation on previous geodesics. For a general domain and metric on it, the geodesic does not always exist; think about a nonconvex domain, equipped with the Euclidean metric. If it exists, it may not be the only one. Some of these possible domains and metrics are discussed in \cite{KeenLa2007}.

In 1916, \textbf{Georg A.~Pick} (1859--1942) connected the Schwarz lemma and the Poincar\'{e} metric in the so-called \emph{Schwarz-Pick lemma}. Observe that assumption about centrality condition $f(0)=0$ is not necessary.

\begin{theorem}[Schwarz-Pick lemma]
\label{izr:sch.pick}
Assume $f\in\hol{\DD}{\DD}$.
\begin{enumerate}
\item Then
\begin{equation}
\label{ena:sch.pick}
\rho(f(p),f(q))\leq\rho(p,q)
\end{equation}
for every $p,q\in\DD$ and
\begin{equation}
\label{ena:sch.pick.1}
|f'(z)| \leq \frac{1-|f(z)|^2}{1-|z|^2}
\end{equation}
for every $z\in\DD$.
\item If $p\neq q$ exist such that the equality in \eqref{ena:sch.pick} is valid or such $z_0$ exists that the equality in \eqref{ena:sch.pick.1} is valid, then $f\in\Aut{\DD}$.
\end{enumerate}
\end{theorem}

Choose arbitrary $f\in\hol{\DD}{\DD}$ and arbitrary points $p,q\in\DD$. Define $F(z):=(\varphi_{f(p)}\circ f\circ \varphi_{-p})(z)$. Then $F\in\hol{\DD}{\DD}$, $F(0)=0$ and
\[
    F'(0)=\frac{f'(p)(1-|p|^2)}{1-|f(p)|^2}.
\]
According to the Schwarz lemma we have $|\varphi_{f(p)}(f(z))|\leq|\varphi_p(z)|$. This is equivalent to \eqref{ena:sch.pick}. If $f(0)=0$, then at \eqref{ena:sch.pick} we get $\rho(0,f(z))\leq\rho(0,z)$, which is equivalent to $|f(z)|\leq|z|$ and at \eqref{ena:sch.pick.1} we get $|f'(0)|\leq1$. Therefore a part (a) of the Schwarz lemma is equivalent to a part (a) of the Schwarz-Pick lemma. To prove that parts (b) of both lemmas are equivalent, note that if $f\in\Aut{\DD}$ and $f(0)=0$, then $f$ is rotation.

\begin{figure}[h!]
\subfigure[][Unit disc with geodesics through $0$ and some balls with center $0$ in Poincar\'{e} metric.]
{\label{fig:disc1}
\begin{tikzpicture}[scale=2.5]
\draw[fill=gray!25,dashed] (0,0) circle (1);
\draw[fill=white] (0,0) circle (0.905148);
\draw[fill=gray!25] (0,0) circle (0.848284);
\draw[fill=white] (0,0) circle (0.761594);
\draw[fill=gray!25] (0,0) circle (0.635149);
\draw[fill=white] (0,0) circle (0.462117);
\draw[fill=gray!25] (0,0) circle (0.244919);
\foreach \i in {0,...,7}
{
\draw[line width=0.25mm] ({-cos(22.5*\i)},{-sin(22.5*\i)})--({cos(22.5*\i)},{sin(22.5*\i)});
}
\end{tikzpicture}
}
\hfill
\subfigure[][Unit disc with geodesics through $(1+\ie)/2$ and balls with center $(1+\ie)/2$ in Poincar\'{e} metric with the same radii as in Figure \ref{fig:disc1}.]
{\label{fig:disc2}
\begin{tikzpicture}[scale=2.5]
\draw[fill=gray!25,dashed] (0,0) circle (1);
\draw[fill=white]
plot [smooth] coordinates{(0.783998,0.588484)(0.777041,0.598397)(0.770087,0.607983)(0.763131,0.61727)(0.756164,0.626286)(0.74918,0.635054)(0.74217,0.643596)(0.735127,0.651934)(0.72804,0.660085)(0.720901,0.668068)(0.713699,0.675899)(0.706424,0.683593)(0.699065,0.691164)(0.69161,0.698625)(0.684046,0.70599)(0.67636,0.713269)(0.668537,0.720475)(0.660564,0.727618)(0.652423,0.734707)(0.644097,0.741753)(0.635567,0.748765)(0.626813,0.75575)(0.617813,0.762717)(0.608542,0.769674)(0.598975,0.776628)(0.589083,0.783585)(0.578833,0.790551)(0.568193,0.797531)(0.557122,0.804529)(0.545579,0.811548)(0.533516,0.818591)(0.520881,0.825656)(0.507615,0.832743)(0.493651,0.839847)(0.478916,0.846959)(0.463326,0.854069)(0.446787,0.861158)(0.42919,0.868204)(0.410414,0.875174)(0.390319,0.882023)(0.368746,0.888695)(0.345514,0.895112)(0.320413,0.901173)(0.293204,0.906745)(0.263613,0.91165)(0.231326,0.915659)(0.195986,0.918462)(0.157192,0.919654)(0.114495,0.918695)(0.0674106,0.914867)(0.0154354,0.907213)(-0.0419142,0.89446)(-0.105026,0.87492)(-0.174055,0.846376)(-0.248699,0.805964)(-0.32782,0.750096)(-0.408885,0.674518)(-0.487201,0.574688)(-0.555077,0.446743)(-0.60135,0.289364)(-0.612151,0.106497)(-0.574008,-0.0900279)(-0.479354,-0.28027)(-0.332006,-0.4406)(-0.148208,-0.551893)(0.0488301,-0.606449)(0.237382,-0.608913)(0.403053,-0.571656)(0.539776,-0.508873)(0.647618,-0.432699)(0.729925,-0.351841)(0.791176,-0.271797)(0.835768,-0.195655)(0.867503,-0.124899)(0.889457,-0.0600319)(0.904027,-0.00100539)(0.913045,0.052517)(0.917896,0.101)(0.919622,0.144947)(0.919005,0.184852)(0.916632,0.221173)(0.912946,0.254327)(0.908279,0.284683)(0.902885,0.312569)(0.896954,0.338269)(0.890631,0.362033)(0.884025,0.384078)(0.877222,0.404594)(0.870283,0.423746)(0.863256,0.441679)(0.856176,0.458521)(0.84907,0.474382)(0.841957,0.489362)(0.83485,0.503546)(0.827758,0.517012)(0.820685,0.529828)(0.813636,0.542055)(0.80661,0.553747)(0.799606,0.564953)(0.792621,0.575718)(0.785651,0.586079)};
\draw[fill=gray!25]
plot [smooth] coordinates{(0.773559,0.581472)(0.766958,0.591375)(0.760331,0.600945)(0.753672,0.610212)(0.746978,0.619201)(0.740242,0.627935)(0.733459,0.636436)(0.726622,0.644724)(0.719724,0.652817)(0.712756,0.660731)(0.705709,0.668484)(0.698574,0.676089)(0.691341,0.683559)(0.683999,0.690908)(0.676536,0.698147)(0.668939,0.705288)(0.661196,0.71234)(0.653291,0.719312)(0.645209,0.726215)(0.636934,0.733055)(0.628446,0.739841)(0.619726,0.746579)(0.610753,0.753276)(0.601504,0.759936)(0.591952,0.766566)(0.58207,0.773167)(0.571827,0.779744)(0.561189,0.786298)(0.550121,0.792829)(0.538579,0.799336)(0.52652,0.805815)(0.513893,0.812262)(0.500641,0.818666)(0.486703,0.825017)(0.472008,0.831297)(0.456478,0.837486)(0.440026,0.843554)(0.422553,0.849466)(0.403948,0.855172)(0.384087,0.860613)(0.362828,0.86571)(0.340013,0.870365)(0.315464,0.87445)(0.288983,0.877804)(0.260345,0.880219)(0.229307,0.88143)(0.195599,0.881096)(0.158938,0.878781)(0.119028,0.873923)(0.0755847,0.8658)(0.0283653,0.853489)(-0.0227811,0.835814)(-0.077827,0.811293)(-0.136435,0.778091)(-0.19776,0.734002)(-0.260175,0.6765)(-0.320918,0.602931)(-0.37571,0.510948)(-0.418501,0.399293)(-0.441618,0.26895)(-0.436722,0.12445)(-0.39681,-0.0253116)(-0.318945,-0.167684)(-0.206448,-0.288842)(-0.0689731,-0.377642)(0.0800355,-0.428765)(0.227223,-0.443453)(0.362431,-0.427792)(0.479871,-0.389982)(0.577638,-0.33804)(0.656458,-0.27854)(0.718455,-0.216271)(0.76625,-0.154408)(0.802437,-0.0948768)(0.829326,-0.0387201)(0.848863,0.0135975)(0.862631,0.0619728)(0.87189,0.106514)(0.877634,0.147444)(0.880639,0.185039)(0.88151,0.219593)(0.880718,0.251395)(0.878629,0.280719)(0.87553,0.307816)(0.871643,0.332916)(0.867143,0.356226)(0.862165,0.377929)(0.856818,0.39819)(0.851183,0.417154)(0.845327,0.434951)(0.839301,0.451695)(0.833145,0.467489)(0.826889,0.482423)(0.820557,0.496579)(0.814167,0.510027)(0.807733,0.522834)(0.801262,0.535056)(0.794763,0.546746)(0.788239,0.557951)(0.781692,0.568712)(0.775122,0.579069)};
\draw[fill=white]
plot [smooth] coordinates{(0.756289,0.570679)(0.74674,0.586077)(0.737002,0.600675)(0.72707,0.614569)(0.716935,0.627843)(0.70658,0.640571)(0.695983,0.652816)(0.685116,0.664636)(0.673947,0.676078)(0.662436,0.687188)(0.650541,0.698001)(0.638209,0.708551)(0.625383,0.718863)(0.611998,0.728958)(0.597978,0.738853)(0.583237,0.748556)(0.567677,0.758068)(0.551184,0.767381)(0.533626,0.776475)(0.514853,0.785316)(0.494688,0.793851)(0.472924,0.802001)(0.449323,0.809652)(0.423604,0.816647)(0.39544,0.822762)(0.364454,0.827688)(0.330212,0.830998)(0.292224,0.832094)(0.249961,0.830146)(0.202882,0.824002)(0.150511,0.812063)(0.0925893,0.792134)(0.0293477,0.761259)(-0.0379992,0.715605)(-0.10647,0.650549)(-0.170229,0.561312)(-0.219506,0.444634)(-0.240546,0.301915)(-0.218394,0.14309)(-0.143616,-0.0118283)(-0.020009,-0.137774)(0.133933,-0.215585)(0.293158,-0.240575)(0.437159,-0.221608)(0.55542,-0.173538)(0.646159,-0.1103)(0.712468,-0.0419077)(0.759097,0.0256048)(0.790702,0.0891253)(0.811167,0.147363)(0.823496,0.200046)(0.829922,0.247415)(0.832074,0.289938)(0.831126,0.328154)(0.827925,0.362596)(0.823078,0.393755)(0.817021,0.422069)(0.81007,0.447918)(0.802451,0.471632)(0.794326,0.493494)(0.785811,0.513744)(0.776986,0.532592)(0.767905,0.550214)(0.758604,0.566764)};
\draw[fill=gray!25]
plot [smooth] coordinates{(0.727796,0.554906)(0.722782,0.564385)(0.717621,0.573543)(0.712318,0.582402)(0.706878,0.59098)(0.701304,0.599296)(0.695597,0.607366)(0.689758,0.615206)(0.683785,0.622828)(0.677676,0.630247)(0.671427,0.637473)(0.665033,0.644517)(0.658491,0.651389)(0.651792,0.658098)(0.64493,0.664649)(0.637896,0.671051)(0.630681,0.677309)(0.623274,0.683426)(0.615664,0.689407)(0.607838,0.695254)(0.599782,0.700969)(0.591481,0.706551)(0.582919,0.711999)(0.574077,0.71731)(0.564937,0.72248)(0.555478,0.727502)(0.545676,0.73237)(0.535507,0.737071)(0.524944,0.741594)(0.513959,0.745921)(0.502521,0.750033)(0.490597,0.753907)(0.47815,0.757514)(0.465143,0.76082)(0.451535,0.763784)(0.437282,0.766361)(0.422339,0.768491)(0.406659,0.77011)(0.390193,0.771137)(0.372891,0.77148)(0.354705,0.771028)(0.335588,0.769654)(0.315498,0.767205)(0.294402,0.763504)(0.272281,0.758345)(0.249134,0.751487)(0.224989,0.742654)(0.199912,0.73153)(0.174023,0.71776)(0.147515,0.700948)(0.12067,0.680669)(0.0938904,0.65648)(0.0677214,0.627942)(0.0428783,0.59466)(0.0202659,0.556334)(0.000983143,0.512831)(-0.0136969,0.464274)(-0.0223852,0.411131)(-0.0236821,0.354303)(-0.0163329,0.29517)(0.000581315,0.235576)(0.0274481,0.177721)(0.0639743,0.123962)(0.109132,0.0765416)(0.161232,0.0373138)(0.218118,0.0075162)(0.277437,-0.0123401)(0.336918,-0.0224528)(0.394596,-0.0236024)(0.448948,-0.016965)(0.498933,-0.0039155)(0.543962,0.0141407)(0.583821,0.0358881)(0.618575,0.0601724)(0.648483,0.0860307)(0.673918,0.112693)(0.695311,0.139566)(0.713106,0.166209)(0.727736,0.192305)(0.739607,0.217638)(0.749085,0.242069)(0.756499,0.265517)(0.762135,0.287944)(0.766246,0.309342)(0.769047,0.329728)(0.770723,0.34913)(0.771434,0.367588)(0.771315,0.385147)(0.77048,0.401856)(0.769028,0.417765)(0.767042,0.432921)(0.764592,0.447374)(0.761737,0.461169)(0.758528,0.47435)(0.755007,0.486959)(0.751209,0.499034)(0.747165,0.510613)(0.742899,0.521729)(0.738433,0.532414)(0.733783,0.542697)(0.728964,0.552605)};
\draw[fill=white]
plot [smooth] coordinates{(0.681304,0.534029)(0.677836,0.542495)(0.674137,0.550694)(0.670219,0.558637)(0.66609,0.566333)(0.661761,0.573793)(0.657236,0.581025)(0.652522,0.588037)(0.647623,0.594837)(0.64254,0.601431)(0.637275,0.607826)(0.631828,0.614026)(0.6262,0.620036)(0.620387,0.62586)(0.614388,0.631499)(0.608199,0.636956)(0.601816,0.642232)(0.595234,0.647326)(0.588446,0.652237)(0.581447,0.656962)(0.574228,0.661498)(0.566783,0.665839)(0.5591,0.669979)(0.551173,0.673911)(0.542989,0.677623)(0.534539,0.681105)(0.525812,0.684342)(0.516795,0.687319)(0.507477,0.690016)(0.497845,0.692411)(0.487887,0.694481)(0.477592,0.696197)(0.466948,0.697526)(0.455943,0.698432)(0.444569,0.698873)(0.432818,0.698803)(0.420685,0.69817)(0.408168,0.696915)(0.395271,0.694972)(0.382,0.69227)(0.368373,0.688729)(0.354415,0.68426)(0.340161,0.67877)(0.325662,0.672155)(0.310984,0.664307)(0.296214,0.655113)(0.281463,0.644456)(0.266869,0.632222)(0.252598,0.618301)(0.238854,0.602596)(0.225874,0.585028)(0.213933,0.565549)(0.203343,0.544151)(0.194444,0.520879)(0.1876,0.495844)(0.183182,0.469233)(0.181553,0.441324)(0.183043,0.412482)(0.187923,0.383165)(0.196379,0.353907)(0.208483,0.325303)(0.224179,0.297973)(0.243272,0.272532)(0.26543,0.249547)(0.290198,0.229504)(0.31703,0.212773)(0.345321,0.19959)(0.374449,0.190054)(0.403807,0.184128)(0.432842,0.181656)(0.46107,0.182391)(0.488099,0.186018)(0.513626,0.192181)(0.537436,0.200509)(0.559398,0.210634)(0.579447,0.222208)(0.597579,0.23491)(0.61383,0.248454)(0.628271,0.262591)(0.640994,0.277108)(0.652107,0.291829)(0.661724,0.306605)(0.669959,0.32132)(0.676928,0.33588)(0.682742,0.350213)(0.687503,0.364264)(0.69131,0.377992)(0.694253,0.39137)(0.696413,0.40438)(0.697865,0.41701)(0.698676,0.429257)(0.698908,0.441121)(0.698613,0.452606)(0.697841,0.46372)(0.696633,0.47447)(0.695029,0.484868)(0.693061,0.494924)(0.690759,0.504651)(0.688149,0.514061)(0.685254,0.523166)(0.682092,0.531978)};
\draw[fill=gray!25]
plot [smooth] coordinates{(0.607816,0.511763)(0.606387,0.517575)(0.604705,0.523249)(0.602779,0.528783)(0.600621,0.534171)(0.598238,0.539411)(0.595638,0.5445)(0.59283,0.549435)(0.58982,0.554212)(0.586613,0.55883)(0.583217,0.563285)(0.579636,0.567575)(0.575874,0.571696)(0.571935,0.575645)(0.567825,0.579417)(0.563545,0.58301)(0.559099,0.586417)(0.554491,0.589635)(0.549723,0.592657)(0.544797,0.595477)(0.539718,0.598089)(0.534486,0.600485)(0.529107,0.602658)(0.523582,0.604597)(0.517915,0.606294)(0.512111,0.607738)(0.506174,0.608918)(0.50011,0.609822)(0.493924,0.610437)(0.487623,0.610748)(0.481217,0.610742)(0.474715,0.610403)(0.468129,0.609715)(0.461471,0.60866)(0.454756,0.607222)(0.448003,0.605383)(0.441231,0.603125)(0.434464,0.600429)(0.427726,0.597279)(0.421047,0.593658)(0.41446,0.589548)(0.408001,0.584938)(0.40171,0.579814)(0.395631,0.574168)(0.389811,0.567993)(0.384303,0.56129)(0.379161,0.554062)(0.374443,0.546319)(0.370209,0.538079)(0.366521,0.529367)(0.36344,0.520215)(0.361028,0.510668)(0.359343,0.500777)(0.35844,0.490604)(0.358368,0.480221)(0.359169,0.469708)(0.360873,0.459155)(0.363501,0.448657)(0.367062,0.438314)(0.371548,0.428231)(0.376938,0.418511)(0.383197,0.409256)(0.390275,0.400564)(0.398105,0.392523)(0.406613,0.385215)(0.415711,0.378707)(0.425304,0.373054)(0.43529,0.368297)(0.445566,0.364461)(0.456029,0.361556)(0.466576,0.359579)(0.47711,0.358513)(0.487541,0.35833)(0.497784,0.358991)(0.507766,0.360449)(0.517421,0.362651)(0.526695,0.36554)(0.535541,0.369055)(0.543925,0.373133)(0.551817,0.377714)(0.559199,0.382736)(0.566059,0.388142)(0.57239,0.393874)(0.578193,0.399881)(0.583471,0.406114)(0.588233,0.412528)(0.592489,0.419081)(0.596254,0.425736)(0.599543,0.432459)(0.602371,0.439221)(0.604757,0.445994)(0.606719,0.452755)(0.608274,0.459483)(0.609441,0.466159)(0.610236,0.472769)(0.610677,0.479297)(0.61078,0.485733)(0.610561,0.492065)(0.610035,0.498286)(0.609216,0.504387)(0.608117,0.510363)};
\draw[line width=0.25mm]
plot [smooth]
coordinates{(0.,1.)(0.000999001,0.968409)(0.00398406,0.937006)(0.00891972,0.905978)(0.015748,0.875501)(0.0243902,0.845743)(0.034749,0.816857)(0.0467112,0.788981)(0.0601504,0.762235)(0.0749306,0.736721)(0.0909091,0.71252)(0.107939,0.689696)(0.125874,0.668293)(0.144568,0.648335)(0.16388,0.629834)(0.183673,0.612782)(0.203822,0.597162)(0.224205,0.582942)(0.244713,0.570083)(0.265246,0.558536)(0.285714,0.548246)(0.306037,0.539155)(0.326146,0.531199)(0.345978,0.524314)(0.365482,0.518435)(0.384615,0.513496)(0.403341,0.509432)(0.421631,0.50618)(0.439462,0.503678)(0.456817,0.501868)(0.473684,0.500693)(0.490056,0.500099)(0.505929,0.500035)(0.521302,0.500454)(0.536178,0.501311)(0.550562,0.502563)(0.56446,0.504172)(0.577881,0.506103)(0.590835,0.50832)(0.603332,0.510794)(0.615385,0.513496)(0.627005,0.516399)(0.638205,0.51948)(0.649,0.522717)(0.659401,0.526089)(0.669421,0.529579)(0.679076,0.533168)(0.688376,0.536843)(0.697337,0.540589)(0.705969,0.544394)(0.714286,0.548246)(0.722299,0.552135)(0.730022,0.556052)(0.737464,0.559988)(0.744637,0.563935)(0.751553,0.567888)(0.758221,0.571839)(0.764651,0.575783)(0.770852,0.579716)(0.776836,0.583632)(0.782609,0.587529)(0.78818,0.591402)(0.793559,0.595249)(0.798752,0.599067)};
\draw[line width=0.25mm]
plot [smooth] coordinates{(-0.595652,0.803242)(-0.545484,0.741356)(-0.493116,0.686835)(-0.439566,0.639311)(-0.385684,0.598311)(-0.332155,0.563304)(-0.27952,0.533733)(-0.228189,0.509041)(-0.178461,0.48869)(-0.130543,0.472174)(-0.0845673,0.459022)(-0.0406059,0.448808)(0.0013153,0.441144)(0.0412061,0.435687)(0.0791027,0.432132)(0.115061,0.430207)(0.149149,0.429678)(0.181446,0.430337)(0.212033,0.432005)(0.240997,0.434525)(0.268422,0.437761)(0.294395,0.441597)(0.318998,0.44593)(0.342311,0.450674)(0.364409,0.455753)(0.385368,0.461102)(0.405254,0.466664)(0.424134,0.472393)(0.442069,0.478247)(0.459117,0.484191)(0.475332,0.490195)(0.490764,0.496232)(0.50546,0.502282)(0.519466,0.508326)(0.532822,0.514347)(0.545567,0.520333)(0.557737,0.526272)(0.569364,0.532156)(0.580482,0.537976)(0.591118,0.543726)(0.6013,0.549401)(0.611054,0.554997)(0.620403,0.56051)(0.62937,0.565938)(0.637975,0.571279)(0.646238,0.576532)(0.654177,0.581696)(0.661808,0.586771)(0.669149,0.591756)(0.676213,0.596651)(0.683015,0.601458)(0.689568,0.606177)(0.695885,0.610808)(0.701976,0.615353)(0.707853,0.619812)(0.713527,0.624188)(0.719006,0.628482)(0.7243,0.632694)(0.729417,0.636826)(0.734366,0.64088)(0.739154,0.644857)(0.743789,0.648758)(0.748276,0.652586)(0.752623,0.656342)};
\draw[line width=0.25mm]
plot [smooth] coordinates{(-0.707107,-0.707107)(-0.586023,-0.586023)(-0.480979,-0.480979)(-0.388984,-0.388984)(-0.30775,-0.30775)(-0.235493,-0.235493)(-0.170803,-0.170803)(-0.11255,-0.11255)(-0.059819,-0.059819)(-0.0118603,-0.0118603)(0.0319458,0.0319458)(0.0721164,0.0721164)(0.109086,0.109086)(0.143222,0.143222)(0.17484,0.17484)(0.204206,0.204206)(0.231554,0.231554)(0.257085,0.257085)(0.280974,0.280974)(0.303375,0.303375)(0.324422,0.324422)(0.344234,0.344234)(0.362918,0.362918)(0.380566,0.380566)(0.397263,0.397263)(0.413084,0.413084)(0.428095,0.428095)(0.442357,0.442357)(0.455926,0.455926)(0.46885,0.46885)(0.481174,0.481174)(0.492939,0.492939)(0.504182,0.504182)(0.514938,0.514938)(0.525237,0.525237)(0.535107,0.535107)(0.544576,0.544576)(0.553666,0.553666)(0.562401,0.562401)(0.570801,0.570801)(0.578884,0.578884)(0.586668,0.586668)(0.59417,0.59417)(0.601404,0.601404)(0.608385,0.608385)(0.615125,0.615125)(0.621638,0.621638)(0.627933,0.627933)(0.634023,0.634023)(0.639916,0.639916)(0.645623,0.645623)(0.651152,0.651152)(0.65651,0.65651)(0.661707,0.661707)(0.666749,0.666749)(0.671643,0.671643)(0.676395,0.676395)(0.681011,0.681011)(0.685498,0.685498)(0.68986,0.68986)(0.694102,0.694102)(0.698231,0.698231)(0.702249,0.702249)(0.706162,0.706162)};
\draw[line width=0.25mm]
plot [smooth] coordinates{(0.803242,-0.595652)(0.741356,-0.545484)(0.686835,-0.493116)(0.639311,-0.439566)(0.598311,-0.385684)(0.563304,-0.332155)(0.533733,-0.27952)(0.509041,-0.228189)(0.48869,-0.178461)(0.472174,-0.130543)(0.459022,-0.0845673)(0.448808,-0.0406059)(0.441144,0.0013153)(0.435687,0.0412061)(0.432132,0.0791027)(0.430207,0.115061)(0.429678,0.149149)(0.430337,0.181446)(0.432005,0.212033)(0.434525,0.240997)(0.437761,0.268422)(0.441597,0.294395)(0.44593,0.318998)(0.450674,0.342311)(0.455753,0.364409)(0.461102,0.385368)(0.466664,0.405254)(0.472393,0.424134)(0.478247,0.442069)(0.484191,0.459117)(0.490195,0.475332)(0.496232,0.490764)(0.502282,0.50546)(0.508326,0.519466)(0.514347,0.532822)(0.520333,0.545567)(0.526272,0.557737)(0.532156,0.569364)(0.537976,0.580482)(0.543726,0.591118)(0.549401,0.6013)(0.554997,0.611054)(0.56051,0.620403)(0.565938,0.62937)(0.571279,0.637975)(0.576532,0.646238)(0.581696,0.654177)(0.586771,0.661808)(0.591756,0.669149)(0.596651,0.676213)(0.601458,0.683015)(0.606177,0.689568)(0.610808,0.695885)(0.615353,0.701976)(0.619812,0.707853)(0.624188,0.713527)(0.628482,0.719006)(0.632694,0.7243)(0.636826,0.729417)(0.64088,0.734366)(0.644857,0.739154)(0.648758,0.743789)(0.652586,0.748276)(0.656342,0.752623)};
\draw[line width=0.25mm]
plot [smooth] coordinates{(1.,0.)(0.968409,0.000999001)(0.937006,0.00398406)(0.905978,0.00891972)(0.875501,0.015748)(0.845743,0.0243902)(0.816857,0.034749)(0.788981,0.0467112)(0.762235,0.0601504)(0.736721,0.0749306)(0.71252,0.0909091)(0.689696,0.107939)(0.668293,0.125874)(0.648335,0.144568)(0.629834,0.16388)(0.612782,0.183673)(0.597162,0.203822)(0.582942,0.224205)(0.570083,0.244713)(0.558536,0.265246)(0.548246,0.285714)(0.539155,0.306037)(0.531199,0.326146)(0.524314,0.345978)(0.518435,0.365482)(0.513496,0.384615)(0.509432,0.403341)(0.50618,0.421631)(0.503678,0.439462)(0.501868,0.456817)(0.500693,0.473684)(0.500099,0.490056)(0.500035,0.505929)(0.500454,0.521302)(0.501311,0.536178)(0.502563,0.550562)(0.504172,0.56446)(0.506103,0.577881)(0.50832,0.590835)(0.510794,0.603332)(0.513496,0.615385)(0.516399,0.627005)(0.51948,0.638205)(0.522717,0.649)(0.526089,0.659401)(0.529579,0.669421)(0.533168,0.679076)(0.536843,0.688376)(0.540589,0.697337)(0.544394,0.705969)(0.548246,0.714286)(0.552135,0.722299)(0.556052,0.730022)(0.559988,0.737464)(0.563935,0.744637)(0.567888,0.751553)(0.571839,0.758221)(0.575783,0.764651)(0.579716,0.770852)(0.583632,0.776836)(0.587529,0.782609)(0.591402,0.78818)(0.595249,0.793559)(0.599067,0.798752)};
\draw[line width=0.25mm]
plot [smooth] coordinates{(0.960304,0.278954)(0.944255,0.274667)(0.927799,0.271035)(0.910975,0.268096)(0.893822,0.265888)(0.876386,0.264445)(0.858717,0.263798)(0.840866,0.263974)(0.822892,0.264994)(0.804851,0.266875)(0.786807,0.269629)(0.768821,0.273262)(0.750958,0.277772)(0.733282,0.283154)(0.715855,0.289395)(0.69874,0.296476)(0.681998,0.304372)(0.665685,0.313052)(0.649856,0.322481)(0.63456,0.332618)(0.619842,0.343419)(0.605745,0.354835)(0.592302,0.366814)(0.579544,0.379304)(0.567494,0.392248)(0.556173,0.405591)(0.545593,0.419275)(0.535761,0.433246)(0.526681,0.447447)(0.51835,0.461824)(0.510762,0.476327)(0.503906,0.490904)(0.497768,0.50551)(0.492329,0.5201)(0.48757,0.534632)(0.483467,0.549069)(0.479995,0.563377)(0.477128,0.577523)(0.474838,0.59148)(0.473097,0.605223)(0.471875,0.618729)(0.471143,0.631981)(0.470873,0.644961)(0.471034,0.657656)(0.471599,0.670055)(0.472539,0.68215)(0.473828,0.693934)(0.475439,0.705402)(0.477346,0.716552)(0.479527,0.727382)(0.481956,0.737892)(0.484614,0.748083)(0.487477,0.757958)(0.490528,0.767521)(0.493746,0.776775)(0.497116,0.785726)(0.500619,0.794379)(0.504242,0.802741)(0.507968,0.810817)(0.511786,0.818614)(0.515683,0.82614)(0.519646,0.833402)(0.523666,0.840408)(0.527733,0.847164)};
\draw[line width=0.25mm]
plot [smooth] coordinates{(0.902369,0.430964)(0.892689,0.426519)(0.882671,0.422283)(0.872316,0.418279)(0.861624,0.41453)(0.850598,0.41106)(0.839242,0.407895)(0.827564,0.40506)(0.81557,0.40258)(0.803272,0.400483)(0.790683,0.398795)(0.777817,0.397543)(0.764692,0.396752)(0.751329,0.396449)(0.737751,0.396659)(0.723982,0.397405)(0.710052,0.398711)(0.69599,0.400596)(0.68183,0.403081)(0.667608,0.406181)(0.65336,0.409911)(0.639127,0.414281)(0.62495,0.4193)(0.610872,0.424972)(0.596936,0.431297)(0.583186,0.438274)(0.569667,0.445895)(0.556424,0.454148)(0.5435,0.46302)(0.530937,0.472491)(0.518777,0.482538)(0.507058,0.493135)(0.495818,0.504253)(0.485091,0.515857)(0.474906,0.527912)(0.465291,0.54038)(0.45627,0.553219)(0.447863,0.566388)(0.440086,0.579843)(0.432951,0.59354)(0.426465,0.607435)(0.420632,0.621483)(0.415453,0.635639)(0.410924,0.649862)(0.407038,0.66411)(0.403785,0.678342)(0.40115,0.692521)(0.399119,0.70661)(0.397673,0.720576)(0.396792,0.734387)(0.396452,0.748014)(0.396631,0.761432)(0.397305,0.774617)(0.398446,0.787548)(0.40003,0.800207)(0.402029,0.812578)(0.404417,0.824647)(0.407167,0.836404)(0.410254,0.84784)(0.413651,0.858947)(0.417334,0.869721)(0.421278,0.88016)(0.425459,0.89026)(0.429855,0.900023)};
\draw[line width=0.25mm]
plot [smooth] coordinates{(0.848786,0.528737)(0.84209,0.524661)(0.835146,0.520628)(0.827948,0.51665)(0.820487,0.512736)(0.812757,0.508898)(0.80475,0.505148)(0.796459,0.501498)(0.787878,0.497964)(0.779001,0.494561)(0.769822,0.491303)(0.760335,0.48821)(0.750537,0.485298)(0.740424,0.482589)(0.729992,0.480101)(0.719241,0.477857)(0.70817,0.47588)(0.69678,0.474194)(0.685073,0.472824)(0.673054,0.471796)(0.660728,0.471137)(0.648105,0.470874)(0.635194,0.471035)(0.622007,0.471651)(0.608562,0.47275)(0.594875,0.474361)(0.580969,0.476513)(0.566867,0.479236)(0.552596,0.482557)(0.538187,0.486502)(0.523675,0.491098)(0.509096,0.496368)(0.49449,0.502333)(0.479901,0.509011)(0.465376,0.516418)(0.450963,0.524566)(0.436714,0.533462)(0.422682,0.543109)(0.408922,0.553506)(0.39549,0.564646)(0.382442,0.576518)(0.369836,0.589105)(0.357726,0.602382)(0.346167,0.616322)(0.335211,0.63089)(0.324906,0.646048)(0.315299,0.661751)(0.306432,0.67795)(0.29834,0.694593)(0.291057,0.711621)(0.284608,0.728977)(0.279013,0.746598)(0.274288,0.764421)(0.270439,0.782383)(0.26747,0.800419)(0.265375,0.818466)(0.264145,0.836462)(0.263764,0.854348)(0.264211,0.872067)(0.265461,0.889565)(0.267484,0.906791)(0.270247,0.9237)(0.273713,0.94025)(0.277843,0.956403)};
\end{tikzpicture}
}
\end{figure}

\section{Inner distances}

What makes the Poincar\'{e} distance exceptional? We could, for example, introduce $\mu(p,q):=|\varphi_{q}(p)|$, which is a distance function on $\DD$ with all the properties as $\rho$ in the Schwarz-Pick lemma. This distance function is called \emph{M\"{o}bius distance}. But there is a crucial difference between those distances: in the Poincar\'{e} distance the boundary is infinitely far away from every point and M\"{o}bius distance seemingly does not have that property. On Figure \ref{fig:disc1} six concentric discs with a center $0$ can be observed. These are balls in the Poincar\'{e} distance with radii $0.5,1,1.5,2,2.5,3$. With the increasing of radii, circles are dense in the neighborhood of the boundary of a disc. This is even more evident if we choose the center of balls near the boundary, as in Figure \ref{fig:disc2}. It can be observed that closed balls in the Poincar\'{e} metric are compact. From the explicit expression for $\rho$ we can prove that every Cauchy sequence with respect to $\rho$ is convergent in $\DD$. We say that $(\DD,\rho)$ is \emph{a complete metric space}. Is there a connection among the infiniteness of a boundary, compactness of closed balls and completeness of a metric space? This question is dealt with in this section.

Let $\gamma\colon[0,1]\rightarrow\Omega$ be a piecewise $\Cc^n$-path from $x$ to $y$ and $\delta:=\{0=t_0<t_1<\cdots<t_k=1\}$ partition of $[0,1]$ on $k$ pieces. \np{Length} of $\gamma$ in space $\left(\Omega,d_\Omega\right)$ is defined by
\[
      L_{d_\Omega}(\gamma) := \sup_\delta \sum_{n=1}^k d_\Omega(\gamma(t_{n-1}),\gamma(t_n)).
\]
We call $d_\Omega^{\;i}(x,y):=\inf L_{d_\Omega}(\gamma)$, where the infimum goes through all piecewise $\Cc^n$-paths $\gamma$ from $x$ to $y$, \np{inner pseudodistance}. It is not difficult to prove that this is indeed a pseudodistance. Because it is always $d_\Omega(x,y)\leq L_{d_\Omega}(\gamma)$, it follows $d_\Omega(x,y)\leq d_\Omega^{\;i}(x,y)$. If the opposite inequality is valid, then we call it $d_\Omega$ \np{inner}. In that case we have $d_\Omega=d_\Omega^{\;i}$.

Let us return to the pseudodistance $d_\Omega(p,q)$, generating with \eqref{hermet}. Let $\gamma\colon[0,1]\rightarrow\Omega$ be a piecewise $\Cc^1$-path from $p$ to $q$. Because we have
\[
      \sum_{n=1}^k d_\Omega(\gamma(t_{n-1}),\gamma(t_n)) \leq \sum_{n=1}^k \int_{t_{n-1}}^{t_n} \sqrt{2(\lambda\circ\gamma)}|\dot{\gamma}|\dif{t} =
      L_{ds_\Omega^2}(\gamma)
\]
for every partition $\delta$, it follows $L_{d_\Omega}(\gamma)\leq L_{ds_\Omega^2}(\gamma)$. Therefore $d_\Omega^{\;i}(p,q)\leq d_\Omega(p,q)$. We have shown that $d_\Omega$ is inner.

We denote the ball with the center $x\in\Omega$ and its radius $r>0$ with $B_{d_\Omega}(x,r)$. The closed ball will be $\overline{B}_{d_\Omega}(x,r)$.

\begin{proposition}
Assume that $d_\Omega$ is a continuous inner distance. Then $d_\Omega$ is equivalent to the Euclidean topology on $\Omega$.
\end{proposition}

\begin{proof}
Choose an arbitrary $x\in\Omega$. Assume $d_\Omega\colon\{x\}\times\Omega\rightarrow[0,\infty)$ is a continuous function. The set $[0,r)\subset[0,\infty)$ is open. Because $B_{d_\Omega}(x,r)=\left(\pr_2\circ d_\Omega^{-1}\right)\left([0,r)\right)$ where $\pr_2$ is a projection to the second component, every $d_\Omega$-ball is open in the Euclidean topology.

Conversely, we will prove that every open set in $\Omega$ is open on $d_\Omega$. Let $U\subset\Omega$ be an arbitrary neighborhood of a point $x\in\Omega$. We must show that $r>0$ exists such that $B_{d_\Omega}(x,r)\subset U$. Choose a relatively compact neighborhood $U'\subset U$ of a point $x$. Define $r:=d_\Omega(x,\partial U')=\inf_{y\in\partial{U'}}d_\Omega(x,y)$. Because $d_\Omega$ is an inner distance, for every point $y\in B_{d_\Omega}(x,r)$ exists a piecewise $\Cc^n$-path $\gamma$ from $x$ to $y$ such that $L_{d_\Omega}(\gamma)<r$. This means that $\gamma\subset B_{d_\Omega}(x,r)$. Hence $B_{d_\Omega}(x,r)\subset U'$, because contrary, for $y\in U\setminus\overline{U'}$ there will be $x'\in\partial U'$ such that $r+d_\Omega(x',y)\leq r$. As this is impossible, the proposition is thus proved.
\end{proof}

Remember that a complete metric space $(X,d_\Omega)$ means that every Cauchy sequence converges in $d_\Omega$. If there is a continuous inner distance, then compactness of closed balls characterizes completeness of a metric space.

\begin{theorem}[Hopf-Rinow]
\label{izr:HR}
Assume that $d_\Omega$ is a continuous inner distance. Then $(\Omega,d_\Omega)$ is a complete metric space if and only if every closed ball $\overline{B}_{d_\Omega}(x,r)$ is compact.
\end{theorem}

\begin{proof}
The easy part of the proof is an implication from compactness of closed balls to  completeness of a space and is valid without the assumption of innerness. Let every closed $d_\Omega$-ball be compact. Because in a metric space every Cauchy sequence has one accumulation point at most and in a compact space every sequence has one accumulation point at least, it follows that $(\Omega,d_\Omega)$ is complete.

Let $(\Omega,d_\Omega)$ be a complete space. Fix $x_0\in\Omega$. Then $r>0$ exists such that $B_{d_\Omega}(x_0,r)$ is relatively compact. If we prove that this is true for all $r>0$, our goal has  been accomplished. Assuming the contrary, set
\[
      r_0 := \sup\left\{r\colon\overline{B}_{d_\Omega}(x_0,r)\;\textrm{is compact}\right\}.
\]
Then the set $\overline{B}_{d_\Omega}(x_0,r_0-\varepsilon)$ is compact for all $\varepsilon>0$. Therefore a sequence $\{y_i\}_{i=1}^n \subset \overline{B}_{d_\Omega}(x_0,r_0-\varepsilon)$ exists such that
\[
      \overline{B}_{d_\Omega}(x_0,r_0-\varepsilon)\subset \bigcup_{i=1}^n B_{d_\Omega}(y_i,\varepsilon).
\]
We will demonstrate that $\{B_{d_\Omega}(y_i,2\varepsilon)\}_{i=1}^n$ is an open cover of $B_{d_\Omega}(x_0,r_0)$. Let us take arbitrary $x\in B_{d_\Omega}(x_0,r_0)\setminus\overline{B}_{d_\Omega}(x_0,r_0-\varepsilon)$. By innerness a piecewise $\Cc^n$-path $\gamma$ exists from $x_0$ to $x$ such that $L_{d_\Omega}(\gamma)<r_0$. Then $t_0\in(0,1)$ and $y_j\in \{y_i\}_{i=1}^n$ exist such that $\gamma(t_0)\in \partial B_{d_\Omega}(x_0,r_0-\varepsilon)$ and $\gamma(t_0)\in B_{d_\Omega}(y_j,\varepsilon)$. Then we have
\begin{flalign*}
L_{d_\Omega}\left(\gamma|_{[t_0,1]}\right) &= L_{d_\Omega}(\gamma)-L_{d_\Omega}\left(\gamma|_{[0,t_0]}\right) \\
                     &< r_0 - (r_0-\varepsilon) = \varepsilon.
\end{flalign*}
This means that $d_\Omega(\gamma(t_0),x)<\varepsilon$ and $d_\Omega(x,y_j)<2\varepsilon$. It follows
\[
    B_{d_\Omega}(x_0,r_0) \subset \bigcup_{i=1}^n B_{d_\Omega}(y_i,2\varepsilon).
\]
Then $n_1\in\NN$, $1\leq n_1 \leq n$ exists such that $\overline{B}_{d_\Omega}(y_{n_1},r_0/2)$ is not compact (we take $\varepsilon = r_0/2$). Set $r_1 := \sup\{r\colon\overline{B}_{d_\Omega}(y_{n_1},r)\;\textrm{is compact}\}$. We inductively continue this process as above. This is how a sequence of points $y_{n_k}\in B_{d_\Omega}(y_{n_{k-1}},r_02^{1-k})$ is obtained, where $\overline{B}_{d_\Omega}(y_{n_k},r_02^{-k})$ is not compact for every $k\in\NN$. The nonconvergent sequence $\{y_{n_k}\}$ is Cauchy, which is in contradiction with the assumption of the completeness of domain $\Omega$. The theorem is therefore proved.
\end{proof}

Let there be a domain $\Omega\subseteq\CC$ and let us choose an arbitrary point $x\in\Omega$. The mapping $\gamma\colon[0,1)\rightarrow\Omega$ is a piecewise $\Cc^n$-path from $x$ to $y\in\partial\Omega\cup\{\infty\}$ if for every $t_0\in(0,1)$ mapping $\gamma|_{[0,t_0]}$ is a $\Cc^n$-path, $\gamma(0)=x$ and $\lim_{t\to1}\gamma(t)=y$. Domain $\Omega$ is \np{b-complete} with respect to the distance $d_\Omega$ if for arbitrary points $x\in\Omega$, $y\in\partial\Omega\cup\{\infty\}$ and for an arbitrary piecewise $\Cc^n$-path $\gamma$ from $x$ to $y$, it follows $\lim_{t\to1}L_{d_\Omega}\left(\gamma|_{[0,t]}\right)=\infty$. Intuitively speaking, $(\Omega,d_\Omega)$ is b-complete if and only if the boundary is ``infinitely far away'' from every inner point.

\begin{corollary}
\label{cor:complete}
Assume that $d_\Omega$ is a continuous inner distance. Then $(\Omega,d_\Omega)$ is a complete metric space if and only if $(\Omega,d_\Omega)$ is b-complete.
\end{corollary}

\begin{proof}
Assume that $\Omega$ is not b-complete. Then there exists a piecewise $\Cc^n$-path $\gamma\colon[0,1)\rightarrow\Omega$ from $x\in\Omega$ to $y\in\partial\Omega$ such that $\lim_{t\to1}L_{d_\Omega}(\gamma|_{[0,t]})=r$ for some $r>0$. For every sequence $\{y_n\}\subset\gamma([0,1))$, where $y_n\to y$, it follows that $d_\Omega(x,y_n)\leq r$ for every $n\in\NN$. Because the closed ball $\overline{B}_{d_\Omega}(x,r)$ is not compact, Theorem \ref{izr:HR} guarantees that $(\Omega,d_\Omega)$ is not complete.

Assume that $\Omega$ is not a complete metric space. Then a Cauchy sequence $\{x_i\}_{i=1}^\infty\subset\Omega$ exists with the limit $x\in\partial\Omega$. Take arbitrary $\varepsilon\in(0,1)$. Since the sequence is Cauchy, then a subsequence $\{k_i\}\subset\NN$ exists such that $d_\Omega(x_{k_i},x_{k_{i+1}})<\varepsilon^i$. Because $d_\Omega$ is an inner distance, there exist piecewise $\Cc^n$-paths $\gamma_i$ with $\gamma_i(0)=x_{k_i}$ and $\gamma_i(1)=x_{k_{i+1}}$ such that $L_{d_\Omega}(\gamma_i)=d_\Omega(x_{k_i},x_{k_{i+1}})+\varepsilon^i<2\varepsilon^i$. Define a piecewise $\Cc^n$-path $\gamma\colon[0,1)\rightarrow\Omega$ from $x_{k_1}$ to $x$ with $\gamma(t):=\gamma_i(2^{i}(t-1)+2)$ for $t\in[1-2^{1-i},1-2^{-i}]$. Take arbitrary $t_0\in(0,1)$. Then $j\in\NN$ exists such that $t_0\in[1-2^{1-j},1-2^{-j}]$. Therefore
\[
     L_{d_\Omega}\left(\gamma|_{[0,t_0]}\right)<2(\varepsilon+\varepsilon^2+\cdots+\varepsilon^j)<\frac{2\varepsilon}{1-\varepsilon}
\]
for $j>1$. Since $\lim_{t\to1}L_{d_\Omega}\left(\gamma|_{[0,t]}\right)<\infty$, the domain $\Omega$ is not b-complete.
\end{proof}

\section{Ahlfors' generalization of the Schwarz-Pick lemma}

As mentioned in the introduction, Ahlfors' generalization was based on curvature. \np{Gauss curvature} $K_{ds_{\Omega}^2}$ of pseudometric \eqref{hermet} is defined by
\begin{equation}
\label{equ:gauss}
K_{ds_{\Omega}^2}(z) := -\frac{1}{\lambda}\frac{\partial^2}{\partial z \partial \bar{z}} \log\lambda(z)
\end{equation}
for $z\in\Omega\setminus Z(\lambda)$ and $-\infty$ for the rest of the points. A simple calculation shows $K_{d\rho^2} \equiv -1$ for Poincar\'{e} metric \eqref{ena:poi}. It is worth mentioning that this curvature is indeed connected to the Gauss curvature of the Riemann metric on surfaces in real differential geometry. A Hermitian metric $2\lambda(z)|\dif{z}|^2$ is a complex analogue of the Riemann metric $E(x,y)\dif{x}^2 + 2F(x,y)\dif{x}\dif{y} + G(x,y)\dif{y}^2$ in real world. Since $z=x+\ie y$, it is easy to accept that $F(x,y)=0$ and $E(x,y)=G(x,y)=2\lambda(x,y)$, so $ds^2=2\lambda(x,y)(\dif{x}^2+\dif{y}^2)$. Let there be $u\in\Cc^2(\Omega,\RR^+)$, where $\Omega\subset\CC$. Then
\begin{equation}
\label{ena:ukr}
\frac{\partial^2}{\partial z\partial\bar{z}} \log{u} = \frac{1}{4}\left(\frac{\partial^2}{\partial x^2} + \frac{\partial^2}{\partial y^2}\right)\log{u} = \frac{u(u_{xx}+u_{yy})-(u_x^2+u_y^2)}{4u^2}.
\end{equation} If the Gauss curvature of $ds^2$ is calculated (see e.g. \cite[Corollary 10.2.3]{Pres}), we get $K_{ds_{\Omega}^2}=(2\lambda)^{-2}K_{ds^2}$. The curvatures seem to be different, nevertheless, the sign does not change.

An important property of the Gauss curvature is invariance on the pullback, which explicitly means that for an arbitrary $f\in\hol{\Omega_1}{\Omega_2}$ there is
\[
     K_{ds_{\Omega_1}^2}(f(z)) = K_{ds_{\Omega_2}^2}(z)
\]
where $ds_{\Omega_2}^2$ is an arbitrary Hermitian pseudometric on $\Omega_2$ and $ds_{\Omega_1}^2:=f^\ast\left(ds_{\Omega_2}^2\right)$. This can be easily seen from \eqref{equ:gauss}, using the chain rule and $f_{\bar{z}}\equiv0$ since $f$ is holomorphic.

We wish to have weaker assumptions for the function $\lambda(z)$. Assume that $\lambda(z)$ is only continuous function. Then $ds_\Omega^2=2\lambda|\dif{z}|^2$ is a continuous Hermitian metric. A pseudometric $\dsp=2\lap(z)|\dif{z}|^2$ is \np{supporting pseudometric} for $ds^2$ at $z_0\in\Omega$ if there is a neighborhood $U\ni z_0$ in $\Omega$ such that $\lap\in\Cc^2(U,\RR_0^+)$ and $\lap|_U \leq \lambda|_U$ with equality at $z_0$. What seems particularly noteworthy is that we do not need a supporting pseudometric, defined on the whole domain $\Omega$. When a supporting pseudometric exists for a continuous pseudometric, this is defined as local existence, which can change from point to point.

\begin{theorem}[Ahlfors' lemma]
\label{izr:ahlfors}
Let $\Omega$ be a domain with a continuous Hermitian pseudometric $ds_{\Omega}^2$, for which a supporting pseudometric $\dsp$ exists. Assume that $K_{\dsp}|_\Omega\leq L$ for some $L<0$. Then for every $f\in\hol{\DD}{\Omega}$ we have
\begin{equation}
\label{IzrAhlEna}
f^\ast(ds_{\Omega}^2) \leq |L|^{-1} d\rho^2
\end{equation}
where $d\rho^2$ is the Poincar\'{e} metric \eqref{ena:poi}.
\end{theorem}

\begin{proof}
By assumptions there is a continuous Hermitian pseudometric $ds_\Omega^2=2\lambda|\dif{z}|^2$. Define
\[
     ds^2:=|L|f^\ast(ds_{\Omega}^2)=2|L|\lambda(f)|f'|^2|\dif{z}|^2.
\]
Then $ds^2$ is a continuous Hermitian pseudometric on $\DD$. Define $\lambda_1:=|L|\lambda(f)|f'|^2$. The equation \eqref{IzrAhlEna} is equivalent to $ds^2 \leq d\rho^2$.

For every $r\in\RR^+$ define $\mu_r(z) := 2r^2(r^2 - |z|^2)^{-2}$ on $\Dr{r}$. Hence $d\rho^2 = 2\mu_1(z)|\dif{z}|^2$. Define the function $u_r(z):=\lambda_1(z)\mu_r^{-1}(z)$. Hence $ds^2 = u_r d\rho_r^2$. If we show that $u_1\leq1$ on $\DD$, then $ds^2 \leq d\rho^2$.

Let us take arbitrary $r'\in(0,1)$. If we show that $u_{r'}(z)\leq1$ for every $z\in\Dr{r'}$, then $u_1\leq1$ on $\DD$, because with $r'\rightarrow 1$ and fixed $z_0\in\Dr{r'}$ it follows $u_{r'}(z_0)\rightarrow u_1(z_0)$. Since $\lambda_1$ is bounded on $\Dr{r'}$, from $|z|\rightarrow r'$ follows $u_{r'}(z)\rightarrow0$. Function $u_{r'}$ is continuous, hence $z_0\in\Dr{r'}$ exists such that $\max u_{r'}=u_{r'}(z_0)$.

Let there be a supporting pseudometric $\dsp$ for $ds_\Omega^2$ at $f(z_0)$. Then $\dspp:=|L|f^\ast(\dsp)$ is a supporting pseudometric for $ds^2$ at $z_0$, whose curvature is $-1$ at most. Then a neighborhood $U\ni z_0$ and $\lapp(z)\in\Cc^2(U,\RR_0^+)$ exist such that $\lapp|_U \leq \lambda_1|_U$ with equality in $z_0$. Define function
\[
     v_r(z):=\frac{\lapp(z)}{\mu_r^{-1}(z)}=\frac{\lapp(z)}{\lambda_1(z)}u_r(z).
\]
Hence $\max_{z\in U} v_{r'}=u_{r'}(z_0)$.

Although what follows is related to the theory of real functions, it is a crucial element of the proof. Let us have $u\in\Cc^2(\Omega,\RR^+)$, where $\Omega\subset\CC$ is a domain. Assume that a function $u$ reaches its maximum at $(x_0,y_0)\in\Omega$. Because this point is singular, it follows $u_x(x_0,y_0)=u_y(x_0,y_0)=0$. But the point is maximum, so $u_{xx}(x_0,y_0)\leq0$ and $u_{yy}(x_0,y_0)\leq0$. By equation \eqref{ena:ukr} we have
\[
     \frac{\partial^2\log{u}}{\partial z\partial\bar{z}} \Bigl|_{z=x_0+\ie y_0}\Bigr. \leq 0.
\]

Remember that the maximum of function $v_{r'}|_U$ is reached at point $z_0$. Hence
\begin{eqnarray*}
0 \geq \frac{\partial^2 \log{v_{r'}|_U}}{\partial z \partial \bar{z}}\Bigl|_{z_0}\Bigr. &=& \frac{\partial^2 \log{\lapp}}{\partial z \partial \bar{z}}\Bigl|_{z_0}\Bigr. - \frac{\partial^2 \log{\mu_{r'}}}{\partial z \partial \bar{z}}\Bigl|_{z_0}\Bigr. \\
&=& -\lapp(z_0)K_{\dspp}(z) - \mu_{r'}(z_0) \\ &=& \mu_{r'}(z_0) \left(-v_{r'}(z_0)K_{\dspp}(z) - 1\right) \\
&\geq& \mu_{r'}(z_0) (v_{r'}(z_0) - 1).
\end{eqnarray*}
We get $v_{r'}(z_0)\leq1$ and $u_{r'}(z_0)\leq1$. Since $z_0$ is the maximum of $u_{r'}$, it follows $u_{r'}(z)\leq1$ on $\Dr{r'}$.
\end{proof}

In the introduction we promised that Theorem \ref{izr:ahlfors} is original version of the Ahlfors lemma. However, Ahlfors proved his lemma for Riemann surfaces. These are one dimensional complex manifolds, so the proof is essentially the same as one above. Under originality we mean the concept of supporting pseudometric. Most authors prove Ahlfors' lemma without it, because for most applications twice-differentiable Hermitian pseudometrics would suffice.

Assume that $L=-1$ in the Ahlfors lemma. Then we have $f^\ast(ds_{\Omega}^2)\leq d\rho^2$. Therefore we can use inequality \eqref{ena:skrcitev.met} and get
\begin{equation}
\label{ena:pos.ahlfors}
d_\Omega(f(p),f(q))\leq\rho(p,q).
\end{equation}
In the case of domain $\left(\DD,d\rho^2\right)$, we get \eqref{ena:sch.pick} of the Schwarz-Pick lemma. Inequality \eqref{ena:sch.pick.1} is even more easily accessible; in \eqref{IzrAhlEna} a proper metrics is put. It is a theorem from 1962 by M.~Heins that in the case of equality in \eqref{IzrAhlEna} for one point only, it follows that there is equality on the whole domain $\Omega$. This can be considered as the generalization of part (b) of the Schwarz-Pick lemma. Interested reader can find simplified proof due to D.~Minda in \cite[Proposition 1.2.1]{JP}.

\section{Applications}

In this section we prove the theorems mentioned in the introduction. Firstly, we will prove the Bloch theorem and a familiar theorem due to Landau, which drops out the assumption about simple discs. These are also Ahlfors' examples of the applications of his lemma. Next, a complete Hermitian metric is constructed, i.e.~an induced distance generates a complete metric, on domain $\CC\setminus\{0,1\}$, which satisfies the assumptions of the Ahlfors lemma. From that point, we are able to provide a proof of the Little Picard theorem. We use the nature of completeness of a space in studying the size of an image of a disc under a holomorphic mapping, which misses two distinct points. This result, named after Schottky is crucial for proving the Big Picard theorem.

\subsection{The Bloch theorem}

Let there be $\mathscr{B}:=\{f\in\ana{\DD}\colon|f'(0)|=1\}$. Remember that the Bloch theorem guarantees the existence of simple discs with a fixed radius in the image $f(\De)$, where $f\in\mathscr{B}$. Let $\op{B}(f)$ be a supremum of all radii of simple discs in $f(\De)$. We want to show that a constant $B>0$ exists such that $\op{B}(f)\geq B$ for every $f\in\mathscr{B}$.

By $S:=\{z\in\DD\colon f'(z)=0\}$ we denote a set of singular points. According to the open mapping theorem, $\Omega:=f(\DD)$ is a domain and $f(\overline{\DD})\subseteq\overline{\Omega}\subset\CC$. For every point $z\in\Omega$ there is a number $\rho(z)$ such that $\Dr{\rho(z)}(z)$ is the largest simple disc. Therefore $\op{B}(f)=\sup_{z\in\Omega}{\rho(z)}$ and $\op{B}(f)<\infty$. On $\DD$ we define a metric
\begin{equation}
\label{ena:bloch1}
\lambda(z) := \frac{A^2 |f'(z)|^2}{2\rho(f(z))\left(A^2-\rho(f(z))\right)^2},
\end{equation}
where $A$ is a constant, which satisfies $A^2>\op{B}(f)$. Since $\rho$ is a continuous function and $\rho(f(z))=0$ if and only if $z\in S$, then \eqref{ena:bloch1} is a continuous Hermitian metric at nonsingular points. We must care only at singular points. Take arbitrary $z_0\in S$. We know that there is a neighborhood $U'\ni z_0$ on $\DD$, $n\geq1$ and biholomorphic function $\varphi(z)$ on $U'$ such that $f(z)=f(z_0)+\varphi^n(z)$ on $U'$ (see e.g. \cite[Corollary 2.3.7]{BerGay1991}). Then there is a neighborhood $U\subset U'$ of point $z_0$ on $\DD$ such that $\rho(f(z))=|f(z)-f(z_0)|$ on $U$. Then the equation \eqref{ena:bloch1} can be rewritten as
\[
   \lambda(z)=\frac{A^2n^2|\varphi(z)|^{n-2}|\varphi'(z)|^2}{2(A^2-|\varphi(z)|^n)^2}
\]
for $z\in U$. Therefore \eqref{ena:bloch1} is a Hermitian pseudometric in the neighborhoods of singular points.

If we want to use the Ahlfors lemma, we need a supporting pseudometric for \eqref{ena:bloch1}. Take an arbitrary nonsingular point $z_0\in\DD\setminus S$. Then $s_0\in \overline{\DD}$ exists such that the boundary of $\Dr{\rho(f(z_0))}(f(z_0))$ contains a point $f(s_0)$. In the neighborhood $U$ of a point $z_0$ it is $\rho(f(z))\leq|f(z)-f(s_0)|$. On $U$ define a Hermitian metric
\begin{equation}
\label{ena:bloch2}
\lap(z) := \frac{A^2|f'(z)|^2}{2|f(z)-f(s_0)|\left(A^2-|f(z)-f(s_0)|\right)^2}.
\end{equation}
The inequality $\lap(z)\leq\lambda(z)$ will be satisfied on $U$ if $x(A^2-x)^2$ is an increasing function on $[0,\op{B}(f)]$. Since $\lap(z_0)=\lambda(z_0)$, metric \eqref{ena:bloch2}, which has constant curvature $-1$, will be supporting for \eqref{ena:bloch1} at $z_0$. A quick calculation shows that the function is increasing on $[0,A^2/3]$. Therefore the metric is supporting if $A^2>3\op{B}(f)$.

Let there be $f(0)=z_0$. By assumption $|f'(0)|=1$, the upper bounds combined with the Ahlfors lemma give
\[
   3\op{B}(f)<A^2\leq4\rho(z_0)(A^2-\rho(z_0))^2\leq 4\op{B}(f)(A^2-\op{B}(f))^2.
\]
Pushing $A^2$ toward $3\op{B}(f)$, we get $\op{B}(f)\geq \sqrt{3}/4$. Hence $B\geq \sqrt{3}/4$.

\textbf{Edmund G.~H.~Landau} (1877--1938) dropped the assumption about simple discs in the Bloch theorem.

\begin{theorem}
\label{thm:landau1}
Assume $f\in\ana{\DD}$ and $|f'(0)|=1$. Then a universal constant $L>0$ exists such that in the image $f(\DD)$ a disc with the radius $R\geq L$ exists.
\end{theorem}

\begin{proof}
Proving this theorem is very similar to proving the Bloch theorem. Let there be a real and positive function $\rho(z)$ such that $\Dr{\rho(z)}(z)$ is the largest disc in $\Omega:=f(\DD)$. Define $\op{L}(f):=\sup_{z\in\Omega}{\rho(z)}$. Since we are not dealing with singular points, we take metrics \[
    \lambda(z):=\frac{1}{2}\left(\rho(z)\log\frac{C}{\rho(z)}\right)^{-2} \;\; \textrm{and} \;\;
    \lap(z):=\frac{1}{2}\left(|z-s_0|\log\frac{C}{|z-s_0|}\right)^{-2}
\]
on $\Omega$. The metric $\lap(z)$ is defined on a neighborhood $U$ of a point $z_0\in\Omega$ and $s_0\in\partial\Dr{\rho(z_0)}(z_0)\cap\partial\Omega$ where it has constant curvature $-1$. Therefore $\lap$ will be supporting for $\lambda$ at $z_0$ if the inequality $\lap(z)\leq\lambda(z)$ is satisfied on $U$. This will be true if $x\log(Cx^{-1})$ is an increasing function on $[0,\op{L}(f)]$. A function is increasing for $ex<C$, therefore the metric is supporting if $e\op{L}<C$.

Assume $f(0)=z_0$. According to the Ahlfors lemma it follows
\[
    1\leq\left(2\rho(z_0)\log\frac{C}{\rho(z_0)}\right)^2\leq\left(2\op{L}(f)\log\frac{C}{\op{L}(f)}\right)^2.
\]
Pushing $C$ toward $e\op{L}(f)$, we get $\op{L}(f)\geq1/2$ and hence $L\geq1/2$.
\end{proof}

The numbers $B$ and $L$ are called the \np{Bloch} and \np{Landau constant}. It follows from the definitions that $B\leq L$. Landau simplified Bloch's proof in 1926 and estimated $B\geq1/16$. It is possible to prove Bloch's theorem with bound $B\geq \sqrt{3}/4$ without Ahlfors' lemma; see \cite[\S10.1.4]{Rem}. Standard proof of Landau's theorem with bound $L\geq1/16$ could be found in \cite[Proposition 2.7.10]{BerGay1991}. The exact values of Bloch and Landau constants are not known \cite[\S10.1.5]{Rem}.

\subsection{The Little Picard theorem} \label{sec:little.picard}

The Little Picard theorem deals with domain $\CC\setminus\{0,1\}$. Therefore, the question whether a Hermitian metric with curvature, bounded with negative constant exists, is reasonable. If this is so, we can use the Ahlfors lemma.

Introduce $\ppC:=\CC\setminus\{0,1\}$. Define
\[
     A_1(z):= \log\frac{C|z|^2}{1+|z|^2}, \;\; A_2(z):= \log\frac{C|z-1|^2}{2(1+|z|^2)}, \;\; A_3(z):= \log\frac{C}{1+|z|^2}
\]
for a constant $C>9$, which will be determined later. The expressions are well-defined on $\ppC$. We will prove that for
\begin{equation}
\label{ena:met.ppC}
\lambda_{\ppC}(z) := \frac{4(1+|z|^2)}{|z|^2|z-1|^2A_1^2(z)A_2^2(z)A_3^2(z)}
\end{equation}
the metric $ds_{\ppC}^2:=2\lambda_{\ppC}(z)|dz|^2$ is a complete Hermitian metric on $\ppC$ with curvature $K(z):=K_{ds_{\ppC}^2}(z)<-1$.

From \eqref{ena:met.ppC} we can see that in the neighborhood of a point $a\in\{0,1\}$ it is
\[
     \lambda_{\ppC}(z) > \frac{A}{|z-a|^2\log^2\left(C|z-a|\right)}
\]
for a constant $A>0$. Let there be a polar presentation $\gamma(t)=r(t)\exp{(\ie\varphi(t))}+a$ of a piecewise $\Cc^1$-path $\gamma\colon[0,1)\rightarrow\ppC$ from $\gamma(0)$ to $\lim_{t\to1}\gamma(t)\in\{0,1\}$. This means that for $t\to1$ it follows $r(t)\to 0$. For the ``point at infinity'' we take the path $\gamma(t)=r(t)\exp{(\ie\varphi(t))}$, where for $t\to1$ it follows $r(t)\to\infty$. Since
\[
     \int_{0}^{t_0} \frac{2\sqrt{A}|\dot{r}+r\dot{\varphi}|dt}{r|\log{Cr}|} \geq \int_{0}^{t_0} \frac{2\sqrt{A}\dot{r}dt}{r\log{Cr}} =
     \sqrt{A}\log\frac{\log^2Cr(t_0)}{\log^2Cr(0)}\xrightarrow{t_0\to1}\infty,
\]
$\left(\ppC,d_{\ppC}\right)$ is b-complete. By Corollary \ref{cor:complete} metric $ds_{\ppC}^2$ is a complete Hermitian metric.

The corresponding Gauss curvature is
\begin{flalign*} K(z) = &-\frac{|z|^2|z-1|^2A_1^2(z)A_2^2(z)A_3^2(z)}{4(1+|z|^2)^{3}} - \frac{|z-1|^2A_2^2(z)A_3^2(z)(1+|z|^2A_1(z))}{2(1+|z|^2)^{3}}
\\
                         &-\frac{|z|^2A_1^2(z)A_3^2(z)(|z+1|^2+|z-1|^2A_2(z))}{2(1+|z|^2)^{3}} \\
                         &-\frac{|z|^2|z-1|^2A_1^2(z)A_2^2(z)(|z|^2+A_3(z))}{2(1+|z|^2)^{3}}.
\end{flalign*}
It can be derived from the expression above that
\begin{flalign*}
\lim_{z\to0}K(z) &= \lim_{z\to\infty}K(z)=-(1/2)\log^2(C/2)\log^2C<-1, \\
\lim_{z\to1}K(z) &= -(1/4)\log^4(C/2)<-1.
\end{flalign*}
For $r_1,r_2,r_3>0$ let us introduce a domain
\[
    X_{r_1,r_2,r_3}:=\{z\in\CC\colon0<|z|<r_1\;\textrm{or}\;0<|z-1|<r_2\;\textrm{or}\;r_3<|z|\}.
\]
Positive numbers $r_1,r_2,r_3$ exist such that $K(z)<-1$ for $z\in X_{r_1,r_2,r_3}$. Because $\ppC\setminus X_{r_1,r_2,r_3}$ is compact, there is a constant $C>9$ such that $K(z)<-1$ for $z\in\ppC\setminus X_{r_1,r_2,r_3}$. Then $K(z)<-1$ on $\ppC$.

Let there be $r>0$ and $f\in\hol{\Dr{r}}{\ppC}$. Then $f(rz)\in\hol{\DD}{\ppC}$. According to the Ahlfors lemma it follows
\begin{equation}
\label{ena:pos.ahl}
r|f'(0)|\leq \frac{2}{\lambda_{\ppC}(f(0))}.
\end{equation}

We can now prove the Little Picard theorem. Assume that $f$ is an entire function such that $f(\CC)\subseteq\ppC$. Choose an arbitrary point $z_0\in\CC$ and introduce a function $g(z):=f(z+z_0)$. Let there be an increasing and unbounded sequence $\{r_n\}_{n=1}^\infty$ of positive real numbers and $g_n:=g|_{\Dr{r_n}}$. By equation \eqref{ena:pos.ahl} for every $n\in\NN$ it follows
\[
    |f'(z_0)|=|g_n'(0)|\leq\frac{2}{r_n \lambda_{\ppC}(g_n(0))}=\frac{2}{r_n \lambda_{\ppC}(f(z_0))}\xrightarrow{n\to\infty}0,
\]
since $g_n(0)=f(z_0)$ and $g_n'(0)=f'(z_0)$. Hence $f'(z_0)=0$. Because $z_0$ was an arbitrary point, it follows $f'\equiv0$ on $\CC$. This means that $f$ is a constant function.

By using inequality \eqref{ena:pos.ahl} we are able to provide a very easy proof of the following Landau theorem from 1904.

\begin{theorem}
\label{thm:landau2}
Assume that $f\in\hol{\Dr{r}}{\ppC}$ for some $r>0$ and $f'(0)\neq0$. Then there is a constant $C>0$, depending only on $f(0)$ and $f'(0)$ such that $r\leq C$.
\end{theorem}

\begin{proof}
Inequality \eqref{ena:pos.ahl} suggests that a good choice for a constant is
\[
    C=\frac{2}{|f'(0)|\lambda_{\ppC}(f(0))},
\]
which only depends on $f(0)$ and $f'(0)$.
\end{proof}

Assume that $f(z)=a_0+a_1z+a_2z^2+\cdots$ is a power series expansion of $f$ at $0$. Then $f(0)=a_0$ and $f'(0)=a_1$. Theorem \ref{thm:landau2} has the following equivalent form: \emph{if $f$ omits $0$ and $1$ and $a_1\neq0$, then a constant $C(a_0,a_1)>0$ exists such that the convergence radius of $f$ is not greater than $C(a_0,a_1)$}. The story goes that Landau was reluctant to publish the above theorem because he found it too absurd \cite[\S10.2.2]{Rem}.

\subsection{The Schottky theorem}

In 1904, the German mathematician \textbf{Friedrich H.~Schottky} (1851--1935) studied the size of an image of a disc under a holomorphic mapping, which omits two distinct points on $\CC$. In our proof completeness of a metric \eqref{ena:met.ppC} is called for. The Hopf-Rinow characterization with closed balls can be used. The theorem is a bridge between the Little and Big Picard theorems.

\begin{theorem}
\label{izr:schottky}
Let $R$ and $C$ be positive real numbers. Assume that we have $f\in\hol{\Dr{R}}{\ppC}$ such that $|f(0)|<C$. Then for every $r\in(0,R)$ a constant $M$ exists, depending only on $R$, $r$ and $C$ such that $|f(z)|\leq M$ for $|z|\leq r$.
\end{theorem}

\begin{proof}
Take $f$ and $r$ from the theorem and set $a:=f(0)$ and $g(z):=f(Rz)$. Then $g\in\hol{\DD}{\ppC}$ and $g(0)=a$. Define $r':=\rho(0,r/R)$. Let $ds_{\ppC}^2$ be metric \eqref{ena:met.ppC} and the $d_{\ppC}$ corresponding distance. Since $d_{\ppC}$ is complete, every closed ball $B:=\overline{B}_{d_{\ppC}}(a,r')$ is compact. Hence, a constant $M_1$ exists, depending only on $R$, $r$ and $C$ such that $|z-a|<M_1$ for every $z\in B$. For every $f\in\hol{\DD}{\ppC}$ there is by \eqref{ena:pos.ahlfors}
\[
    d_{\ppC}(a,f(z))\leq\rho(0,z)\leq r'
\]
for $|z|\leq r$ and therefore $|f(z)-a|<M_1$. Since $|a|<C$, it follows $|f(z)|<C+M_1$ for $|z|\leq r$. The theorem follows after setting $M:=M_1+C$.
\end{proof}

\subsection{The Big Picard theorem}

While studying the properties of $\holC{\Omega}$, we should introduce the concept of \emph{normal families}: a family $\mathscr{F}\subset\holC{\Omega}$ is normal if every sequence in $\mathscr{F}$ has a convergent subsequence in $\holC{\Omega}$, where convergence is uniformly on compact sets. For $\mathscr{F}$ we say that it is \emph{bounded on $\Omega$} if for every compact set $K\subset\Omega$ a constant $C(K)$ exists such that
\begin{equation}
\label{equ:omejenost}
\sup_{f\in\mathscr{F}}\left(\sup_{z\in K}|f(z)|\right)\leq C(K).
\end{equation}
A common sign for $\sup_{z\in K}|f(z)|$ is $\|f\|_K$. \textbf{Paul A.~A.~Montel} (1876--1975) proved the following theorem in 1907.

\begin{theorem}
\label{izr:montel}
A family $\mathscr{F}\subset\holC{\Omega}$ is bounded on $\Omega$ if and only if it is normal.
\end{theorem}

The difficult part of proving the Montel theorem comes from the implication from a bound to normality, since the opposite direction is quite clear. Assume that a compact set $L\subset\Omega$ exists for which \eqref{equ:omejenost} is not true. Then a sequence $\{f_n\}\subset\mathscr{F}$ exists, such that $\|f_n\|_L\rightarrow\infty$ and does not have a convergent subsequence with the limit $f$, since $\|f_n\|_L-\|f-f_n\|_L \leq \|f\|_L$. The core of the problem is the celebrated \emph{Arzel\'{a}-Ascoli theorem} (details may be found in \cite[\S5.5]{AhlCplx}), which asserts that a family $\mathscr{F}$ of continuous functions on $\Omega$ is relatively compact if and only if the family is equicontinuous and $f(z_0)$ is relatively compact for every $z_0\in\Omega$ and every $f\in\mathscr{F}$. The latter is satisfied since $\mathscr{F}$ is bounded and equicontinuity follows from the Cauchy inequality. By Weierstrass' theorem, which asserts that a family of holomorphic functions is closed in a family of continuous functions, we see that the Arzel\'{a}-Ascoli theorem implies the Montel theorem.

It is useful to expand the definition of normality in the direction that allows uniform convergence on compact sets to $\infty$. A closed family $\mathscr{F}\subset\hol{\Omega_1}{\Omega_2}$ is \np{normal} if every sequence in $\mathscr{F}$ has convergent subsequence or this sequence is \emph{compactly divergent}. This means that for arbitrary compact sets $K\subset\Omega_1$ and $L\subset\Omega_2$ integer $N\in\NN$ exists such that $f_n(K)\cap L = \varnothing$ for all $n>N$.

For the proof of the next theorem we need the classical result (see e.g.~\cite[p.~178]{AhlCplx}) by A.~Hurwitz: \emph{Assume that $\{f_n\}\subset\holC{\Omega}$ is a convergent sequence with the limit $f\in\holC{\Omega}$. If $a\in\CC$ exists such that $a\notin f_n(\Omega)$ for every $n\in\NN$, then $a\notin f(\Omega)$ or $f\equiv a$.}

\begin{theorem}[Normality theorem]
\label{izr:os.krit.nor}
Let there be $a,b\in\CC$, $a\neq b$. Then the family $\mathscr{F}\subseteq\hol{\Omega}{\CC\setminus\{a,b\}}$ is normal for every domain $\Omega\subset\CC$.
\end{theorem}

\begin{proof}
Let there be $\mathscr{F}\subset\hol{\Omega}{\CC\setminus\{0,1\}}$ and $\{f_n\}_{n=1}^\infty\subset\mathscr{F}$ is an arbitrary sequence. It is enough to show that for every point $x\in\Omega$ there is a neighborhood $U\subset\Omega$ such that the family $\{f_n|_U\colon n\in\NN\}$ is normal.

Choose a fixed but arbitrary point $x\in\Omega$. An unbounded sequence $\{f_n(x)\}$ is compactly divergent. If the sequence $\{f_n(x)\}$ is bounded, then according to the Schottky theorem $C>0$ exists such that $f_n(\Dr{r}(x))\subset\Dr{C}$, where such $r>0$ is chosen that $\Dr{r}(x)\subset\Omega$. Since
\[
     \sup_{f\in\{f_n\}}\|f\|_{\overline{\DD}_r(x)} \leq C,
\]
according to the Montel theorem, a subsequence $\{f_{n1}\}\subset\{f_n\}$ exists such that $f_{n1}$ uniformly converges to $f\in\holC{\Dr{r}(x)}$ on compact sets. If $f(\Dr{r}(x))\subset\CC\setminus\{0,1\}$, the goal has been achieved. Therefore, let us have $z_0\in\Dr{r}(x)$ such that $f(z_0)\in\{0,1\}$. Assume that $f\neq f(z_0)$. According to the Hurwitz theorem $N\in\NN$ exists such that $f(z_0)\in f_{n1}(\Dr{r}(x))$ for all $n>N$. Because this is not true, it follows that $f\equiv f(z_0)$. Therefore, the sequence is compactly divergent.
\end{proof}

Montel proved the Normality theorem in 1912, which is why it is sometimes referred to as \emph{the Big Montel theorem}.

The Normality theorem is used to prove the Big Picard theorem. We are going to prove the sharper and not so widely known form of the theorem by \textbf{Gaston M.~Julia} (1893--1978) from 1924. For the formulation of the theorem the following ``cone-shape'' domain is needed
\[
    J(\zeta,\alpha):=\left\{t_1e^{\ie(\varphi_0+t_2\alpha)}\colon t_1\in(0,1),t_2\in(-1/2,1/2),e^{\ie\varphi_0}=\zeta\right\} \subset \DD^\ast.
\]
The domain is a disc section with an angle $\alpha$, which is symmetric on a chord with endpoints $0$ and $\zeta\in\partial\DD$.

\begin{theorem}
Assume that $f\in\holC{\DD^\ast}$ with an essential singularity at $0$. Then $\zeta\in\partial\DD$ exists such that for all $\alpha>0$ function $f$ on $J(\zeta,\alpha)$ takes every value in $\CC$ infinitely often with no more than one exception.
\end{theorem}

\begin{proof}
Assume that for every $\zeta\in\partial\DD$ there is $\alpha_\zeta>0$ such that for some $a,b\in\CC$, $a\neq b$ equations $f(z)=a$ and $f(z)=b$ have finite solutions on $J(\zeta,\alpha_\zeta)$. Since the boundary of a disc is a compact set, a sequence $\{\zeta_n\}_{n=1}^N$ exists such that $\{J(\zeta_n,\alpha_{\zeta_n})\}$ is a finite open cover of $\DD^\ast$ with the previously mentioned property. Then there is $\varepsilon>0$ such that $f(\Dr{\varepsilon}^\ast)\subset\CC\setminus\{a,b\}$.

Let $\mathbb{A}(r_1,r_2):=\{z\in\CC\colon r_1<|z|<r_2\}$ be an annulus. Define the family $\mathscr{F}:=\{f_n(z)|_{\mathbb{A}(1/2,2)}\}$, where $f_n(z):=f(\varepsilon2^{-n}z)$. Since
\[
    \mathscr{F}\subset\hol{\mathbb{A}(1/2,2)}{\CC\setminus\{a,b\}},
\]
the Normality theorem guarantees that $\mathscr{F}$ is a normal family. Then a subsequence $\{f_{n1}\}\subset\{f_n\}$ exists such that $f_{n1}\rightarrow g\in\holC{\mathbb{A}(1/2,2)}$ or $\{f_n\}$ is compactly divergent. In the first case, the sequence $\{f_{n1}|_{\partial\DD}\}$ is uniformly bounded. In the second case, the sequence $\left\{\left(f_n|_{\partial\DD}\right)^{-1}\right\}$ is uniformly bounded. Assume that we are dealing with the first case. Then there exists $M>0$ such that $|f(z)|<M$ for every $|z|=\varepsilon2^{-n1}$. According to the maximum principle $|f(z)|<M$ in the neighborhood of $0$. This means that singularity is removable. This is in contradiction with the assumption of an essential singularity. In the second case, we get a removable singularity for $1/f$ in $0$, which is also a contradiction.
\end{proof}

\section{A glimpse of hyperbolic complex manifolds}

In this final section we briefly describe main properties of hyperbolic complex manifolds. We are especially interested on those properties which are in direct connection with Picard's theorems.

We begin with the notion of \emph{invariant pseudodistances}. These are pseudodistances which can be constructed on the category of complex manifolds and they become isometries for biholomorphic  mappings. In 1967, a Japanese mathematician \textbf{Shoshichi Kobayashi} (1932--2012) constructed one of those pseudodistances. For every $x,y\in M$ and every $f\in\hol{M}{N}$, where $M$ and $N$ are complex manifolds, \emph{Kobayashi pseudodistance} $d_M^K$ has the following properties
\begin{gather}
\label{ena:skrcitev}
d_N^K(f(x),f(y))\leq d_M^K(x,y), \\ \label{ena:skrcitev1}
d_{\DD}^K(x,y)=\rho(x,y).
\end{gather}
Thus $d_M^K$ is an invariant pseudodistance, which coincides with the Poincar\'{e} distance on a disc. Explicit construction is carried out by the so-called \emph{chain of holomorphic discs}
\[	\alpha : \left\{ \begin{array}{l}
                  p=p_0,p_1,\ldots,p_k=q \in M, \\
                  a_1,a_2,\ldots,a_k \in \DD, \\
                  f_1,f_2,\ldots,f_k \in \hol{\DD}{M}
                  \end{array}
         \right.
\]
between $p,q\in M$ where $f_n(0)=p_{n-1}$ and $f_n(a_n)= p_n$ for all $n\in\{1,\ldots,k\}$. Kobayashi pseudodistance is then defined as
\[
    \dk{M}{p}{q}:=\inf_\alpha\left\{\sum_{n=1}^{k}\rho(0,a_n)\right\}.
\]
Thus the construction of $d_M^K$ is in the spirit of Bloch principle. Generally speaking, $d_M^K$ is not a distance.

\begin{example}
We have $d_\CC^K\equiv0$. To see this, take a holomorphic mapping $f(z):=p+ \varepsilon^{-1}(q-p)z$ from $\DD$ into $\CC$, where $p,q\in\CC$ are arbitrary points and $\varepsilon>0$ is an arbitrary small number. Then $f(0)=p$ and $f(\varepsilon)=q$. From \eqref{ena:skrcitev} we get $d_\CC^K(p,q)\leq2\varepsilon$. Because mapping $\exp\colon\CC\rightarrow\pC$ is surjective, it also follows $d_{\pC}^K\equiv0$.
\end{example}

A complex manifold is \emph{hyperbolic} if the Kobayashi pseudodistance becomes a distance and \emph{complete hyperbolic} if $(M,d_M^K)$ is a complete metric space. Hyperbolic manifolds have several important properties, including the fact that the Kobayashi distance is inner, direct product of (complete) hyperbolic manifolds is (complete) hyperbolic and (complete) hyperbolicity is invariant with respect to unramified covering projections. The latter statement combined with Poincar\'{e}-Koebe uniformization theorem asserts that Riemann surface is hyperbolic if and only if its universal cover is a disc. This is in agreement with traditional meaning of hyperbolic Riemann surfaces.

We call pseudodistances, which satisfy the properties \eqref{ena:skrcitev} and \eqref{ena:skrcitev1}, \emph{contractible pseudodistances}. It can be shown that the Kobayashi pseudodistance is the largest among contractible pseudodistances. What is more, if a pseudodistance $d_M$ satisfies $d_M(f(x),f(y))\leq\rho(x,y)$ for all $x,y\in\DD$ and all $f\in\hol{\DD}{M}$, then $d_M\leq d_M^K$. Ahlfors' lemma implies that every planar domain (or more generally every Riemann surface), which carries a complete Hermitian metric of curvature not greater than $-1$, is complete hyperbolic. Therefore, the domain $\CC\setminus\{0,1\}$ is complete hyperbolic. From this it is easy to see that every domain $\Omega\subseteq\CC\setminus\{a,b\}, a\neq b$, is complete hyperbolic; observe that every Cauchy sequence in $\Omega$ is also Cauchy in $\CC\setminus\{a,b\}, a\neq b$. Of course in higher dimensions there exist hyperbolic domains which are not complete hyperbolic. Probably the simplest example is a punctured bidisc $\DD_2\setminus\{(0,0)\}$.

\begin{example}
Denote punctured bidisc by $X$. Because $X$ is bounded, it is hyperbolic. Define the following sequences $a_n:=(0,\alpha_n)$, $b_n:=(\alpha_n,0)$ and $c_{n,m}:=(\alpha_n,\alpha_m)$ where $\{\alpha_n\}\subset\DD$ is a sequence with property $\rho(0,\alpha_n) = 2^{-n}$. Introducing domains $X_1:=\DD\times\DD^\ast \subset X$ and $X_2:=\DD^\ast\times\DD\subset X$ yields
\begin{gather*}
\dk{X}{a_n}{b_n} \leq \dk{X_1}{a_n}{c_{n,n}} + \dk{X_2}{b_n}{c_{n,n}}, \\
\dk{X}{b_n}{a_{n+1}} \leq \dk{X_2}{b_n}{c_{n,n+1}} + \dk{X_1}{a_{n+1}}{c_{n,n+1}}.
\end{gather*}
Define $f_n\in\hol{\DD}{X_1}$ with $f_n(z):=(z,\alpha_n)$ and $g_n\in\hol{\DD}{X_2}$ with $g_n(z):=(\alpha_n,z)$. Then $\dk{X_1}{a_n}{c_{n,n}}\leq\rho(0,\alpha_n)$ and $\dk{X_1}{a_{n+1}}{c_{n,n+1}}\leq\rho(0,\alpha_n)$. Equivalently $\dk{X_2}{b_n}{c_{n,n}}\leq\rho(0,\alpha_n)$ and $\dk{X_2}{b_n}{c_{n,n+1}}\leq\rho(0,\alpha_n)$. Thus $\dk{X}{a_n}{b_n} \leq2^{1-n}$ and $\dk{X}{b_n}{a_{n+1}}\leq2^{1-n}$. Therefore $\{a_n\}_{n=1}^\infty$ is a Cauchy sequence which converges to $(0,0)\notin X$.
\end{example}

Hyperbolicity is closely related to the Little Picard theorem. If we assume $f\in\hol{\CC}{M}$, then we get $d_M^K(f(x),f(y))=0$ by \eqref{ena:skrcitev}. This shows that \emph{every holomorphic map from $\CC$ to a hyperbolic manifold is constant}. The converse of this statement is not true in general; however, it is true on compact complex manifolds in view of a fundamental theorem of Robert Brody from 1978. The image of holomorphic map from $\CC$ is said to be \emph{entire curve}. A complex manifold is said to be \emph{Brody hyperbolic} if all entire curves on it are constants. We prove that $\CC\setminus\{a,b\},a\neq b$ is also Brody hyperbolic.

Our definition of normal families can be adapted to holomorphic mappings between complex manifolds. A complex manifold $M$ is said to be \emph{taut} if  $\hol{N}{M}$ is a normal family for every complex manifold $N$. Taut manifolds are somewhere between complete hyperbolic and hyperbolic manifolds since it can be shown that completeness implies tautness and tautness implies hyperbolicity. The Hopf-Rinow theorem is crucial to prove this assertions. Therefore $\CC\setminus\{a,b\},a\neq b$ is taut domain which implies Normality theorem.

Is there any generalization of the Big Picard theorem in the sense of hyperbolicity? The answer is yes and it goes through \emph{hyperbolic imbeddings}. Let $X$ be a relatively compact domain in complex manifold $M$. If for every $x,y\in\overline{X}$ there exist neighborhoods $U\ni x$ and $V\ni y$ in $M$ such that $\dk{X}{U\cap X}{V\cap X}>0$, then $X$ is hyperbolically imbedded in $M$. Peter Kiernan coined this term in 1973 and proved that if $X$ is hyperbolically imbedded domain in $M$, then every map $f\in\hol{\DD^\ast}{X}$ has an extension to $\tilde{f}\in\hol{\DD}{M}$. Since $\cp{1}\setminus\{0,1,\infty\}$ is biholomorphic to $\CC\setminus\{0,1\}$ it follows from properties of the Hermitian metric, constructed in section \ref{sec:little.picard}, that $\cp{1}\setminus\{0,1,\infty\}$ is hyperbolically imbedded in $\cp{1}$.

We can consider a point in $\cp{1}$ as a hyperplane. Mark L. Green proved: \emph{complement of $2n+1$ hyperplanes in general position in $\cp{n}$ is complete hyperbolic and hyperbolically imbedded in $\cp{n}$}. Main ideas in the proof of this theorem are:
\begin{enumerate}
\item using Nevanlinna theory of meromorphic functions to study entire curves in complements of hyperplanes; here the most useful result is \emph{Borel's lemma: if $f_1+\cdots+f_n\equiv1$ for entire functions $f_1,\ldots,f_n$ without zeros, then at least one function must be a constant}. Observe that this implies Little Picard theorem.
\item extending Brody's theorem to obtaining a criteria for hyperbolically imbedded complements of complex hypersurfaces in compact complex manifolds;
\item connect hyperbolic imbeddings and complete hyperbolicity of such complements; it can be shown that \emph{if a complement of a complex hypersurface is hyperbolically imbedded in a complex manifold, then this complement is complete hyperbolic}.
\end{enumerate}

All results mentioned in this section can be found in Kobayashi's book \cite{Ko2}, still ultimate reference concerning hyperbolic complex spaces. The same author in \cite{Kob2005} offers an excellent introduction to the subject while Krantz's book \cite{Kra2004} has similar approach to the subject as here. The historical aspect of invariant pseudodistances and hyperbolicity are described in \cite{Roy88}. The greatness and beauty of invariant pseudodistances can be found in a comprehensive book \cite{JP}.

\subsection*{Acknowledgements}
I would like to thank my advisor prof.~Franc Forstneri\v{c} for valuable advices and comments about the improvement of the article.

\providecommand{\bysame}{\leavevmode\hbox to3em{\hrulefill}\thinspace}
\providecommand{\MR}{\relax\ifhmode\unskip\space\fi MR }
\providecommand{\MRhref}[2]{%
  \href{http://www.ams.org/mathscinet-getitem?mr=#1}{#2}
}
\providecommand{\href}[2]{#2}

\end{document}